%% file: main.tex
\documentclass[12pt,a4paper,oneside,reqno]{amsproc}
\usepackage{amsmath} %amscd
\usepackage{amssymb,amsfonts,mathrsfs}
\usepackage[colorlinks=true,linkcolor=blue,citecolor=blue]{hyperref}
\usepackage[driver=pdftex,lmargin=2.5cm,rmargin=2.5cm,heightrounded=true,centering]{geometry}
\usepackage{math50}
\usepackage{local}
\usepackage{epsfig}  % want eps figures?
\usepackage{graphicx}  % want eps figures?
\usepackage{xy}

\usepackage{setspace}%%%%
\xyoption{all}

\begin{document}

\begin{spacing}{1.0}%%

\title[The Hartree equations]
{Nondegeneracy of ground states and multiple semiclassical solutions of the Hartree equation for general dimensions}
\author{Guoyuan Chen}
\address{School of Data Sciences, Zhejiang University of Finance \& Economics, Hangzhou 310018, P. R. China}
\email{gychen@zufe.edu.cn}
\newcommand{\optional}[1]{\relax}
\setcounter{secnumdepth}{3}
\setcounter{section}{0} \setcounter{equation}{0}
\numberwithin{equation}{section}

%\thanks{The author was partially supported by National Natural Science Foundation of China (No.11401521 and 11402226).}
%\subjclass[2000]{Primary 58J37; Secondary 58J40; 58J50; 58J05}
\date{\today}
\begin{abstract}
We study nondegeneracy of ground states of the Hartree equation
$$
-\Delta u+u=(I_{2}\ast u^2)u\quad\mbox{ in }\mathbb R^n
$$
where $n=3,4,5$ and $I_2$ is the Newton potential.
As an application of the nondegeneracy result, we use a Lyapunov-Schmidt reduction argument to construct multiple semiclassical solutions to the following Hartree equation with an external potential
$$-\varepsilon^2\Delta u+u+V(x)u=\varepsilon^{-2}(I_{2}\ast u^2)u\quad \mbox{ in }\mathbb R^n.$$
\end{abstract}
\maketitle

\input{Introduction}

%\newpage

\input{Preliminaries}
%\newpage

%\input{Expansion}
%\newpage

\input{Nondegeneracy}

%\newpage

\input{Applications}

%\newpage

%\input{Appendix}
%\newpage

\newcommand{\Toappear}{to appear in}

\bibliography{mrabbrev,cz_abbr2003-0,localbib}

\bibliographystyle{abbrv}
\end{spacing}
\end{document}

%% file: Introduction.tex
\section{Introduction and main results}

We consider the Hartree equation:
\begin{equation}\label{e:ug}
-\Delta u+u=(I_{2}\ast u^2)u\quad \mbox{ in }\mathbb R^n,
\end{equation}
where $I_2(x)=\frac{1}{(n-2)|\mathbb S^{n-1}|}\frac{1}{|x|^{n-2}}$ is the Newton potential in the Euclidean space $\mathbb R^n$. Here $|\mathbb S^{n-1}|$ denotes the $(n-1)$-dimensional volume of $\mathbb S^{n-1}$. Equation (\ref{e:ug}) is equivalent to the Schr\"odinger-Newton equation
\begin{equation}\label{e:sn}
\left\{\begin{array}{rc}
         -\Delta u+u=vu & \mbox{ in }\mathbb R^n, \\
         -\Delta v=u^2 & \mbox{ in }\mathbb R^n.
       \end{array}
\right.
\end{equation}
Equation (\ref{e:ug}) can be considered as a special case of the Choquard equation
\begin{eqnarray}\label{e:gce}
-\Delta u+u=(I_{\alpha}\ast |u|^p)|u|^{p-2}u\quad \mbox{ in }\mathbb R^n,
\end{eqnarray}
where $I_{\alpha}$ denotes the Riesz potential with $\alpha\in (0,n)$ in $\mathbb R^n$, $p>1$. It is clear that (\ref{e:ug}) is the case $\alpha=2$ and $p=2$ in (\ref{e:gce}).
In the present paper, we will restrict our attention to Equation (\ref{e:ug}) (and (\ref{e:sn})).

 The Hartree equation is from various physical models such as H. Fr\"ohlich and S. Pekar's model of the polaron (\cite{Frohlich37,Pekar1954}), Ph. Choquard's model of an electron trapped in its own hole, and the Hartree-Fock theory of one-component plasma (\cite{Lieb1977}). The Hartree equation (or the Sch\"odinger-Newton equation) couples the Sch\"odinger equation of quantum physics together with non-relativistic Newtonian potential (see, for example, \cite{Bahrami2014,DIOSI1984,Jones1995,Penrose1996,Penrose1998,MoPenTod1998,TodMo1999} and the references therein).

Equation (\ref{e:ug}) corresponds to the following functional
\begin{equation}\label{e:functional}
F(u)=\frac{1}{2}\int_{\mathbb R^n}|\nabla u|^2+|u|^2-\frac{1}{4}\int_{\mathbb R^n}(I_{2}\ast |u|^2)|u|^2.
\end{equation}
From the Hardy-Littlewood-Sobolev inequality and the Sobolev imbedding theorem, we have to choose $n=3,4,5$ to ensure that $F$ is well-defined, continuously Fr{\'e}chet differentiable and has a ground state in $H^1(\mathbb R^n)$ (see, for example, \cite{TodMo1999,MorVSJFA2013}).

Many works were devoted to the study of basic properties of the ground states solution to the Choquard equation (\ref{e:gce}) (see i.e. \cite{Lieb1977,Lions19851,Lions19852,TodMo1999,Lenzmann2009,Ma2009,WeiWin2009,cingolani2012,MOROZ2013,Wang2016,Xiang2015} and the references therein).
For our further application to the Hartree equation, we summarize the properties as follows.

Let $n=3,4,5$. If $u\in H^1(\mathbb R^n)$ is a ground state of $F$, then we have the following conclusions:
\begin{itemize}
  \item (\emph{Regularity, symmetry}) $u$ belongs to $L^1(\mathbb R^n)\cap C^{\infty}(\mathbb R^n)$, $u$ is either positive or negative and there exists $x_0\in \mathbb R^n$ and a monotone function $v\in C^{\infty}(0,\infty)$ such that $u(x)=v(x-x_0)$ for all $x\in \mathbb R^n$ (see i.e. \cite[Theorem 3]{MorVSJFA2013});
  \item (\emph{Decay property}) $u$ is exponentially decaying and there exists $c>0$ such that
\begin{equation}\label{e:Uasymptotics}
u(x)=(c+o(1))|x|^{-\frac{n-1}{2}}\exp\left(-\int_{\nu}^{|x|}\sqrt{1-\frac{\nu^{n-2}}{s^{n-2}}}ds\right)\quad \mbox{as }x\to \infty,
\end{equation}
where $\nu^{n-2}=\frac{\Gamma((n-2)/2)}{4\pi^{n/2}}\int_{\mathbb R^n}u^2$ (see i.e. \cite[Theorem 4]{MorVSJFA2013}).
\end{itemize}
Moreover, there exists one and only one radial positive solution to (\ref{e:ug}). Hence up to translations, the set of ground states has a \emph{unique} element (\cite{Wang2016,Ma2009}.

Throughout this paper, we will denote by $U$ the unique positive radial ground state of (\ref{e:ug}).

\subsection{Nondegeneracy}
Our first result is to verify that for $n=3,4,5$, the ground state $U$ is nondegenerate up to translations. To be more precise, we have
\begin{theorem}\label{t:nondegeneracy}
Assume $n=3,4,5$. Let $L$ be the linearised operator at $U$ given by
\begin{equation}\label{e:linearised}
L\varphi:=-\Delta \varphi+\varphi-(I_{2}\ast U^2)\varphi-2(I_{2}\ast (U\varphi))U,\quad \forall \varphi\in H^2(\mathbb R^n).
\end{equation}
Then
$${\rm ker}\,L={\rm span}\{\partial_{x_1} U,\partial_{x_2} U,\cdots,\partial_{x_n} U\}.$$
\end{theorem}
\begin{remark}
(1) When $n=3$, such kind of nondegeneracy results was obtained by \cite{WeiWin2009, Lenzmann2009}.
For the case $n=3$ and $p>2$ close to $2$, the nondegeneracy was also verified by \cite{Xiang2015}.

(2) Here the Sobolev space $H^2(\mathbb R^n)$ consists of real-valued functions. For the complex-valued case, the nondegeneracy can be showed by splitting the linearized operator into real and imaginary parts (see \cite{Lenzmann2009}). Without loss of generality, we restrict our investigations in the real Sobolev spaces.
\end{remark}

For various nonlinear Schr\"odinger equations with local nonlinearities, the nondegeneracy of ground states is widely investigated. It is a key ingredient in the stability analysis of solitary waves and the Lyapunov-Schmidt reduction method of constructing semiclassical solutions. The classical method to obtain the nondegeneracy relies on the Sturm-Liouville theory (see i.e. \cite{Wein1985}). For the equations containing nonlocal terms, the problem may become more involved (see i.e. \cite{Amick&Toland91,FL:UNGS,Frank2016,WeiWin2009,Lenzmann2009}).

For the Hartree equation (\ref{e:ug}) in $\mathbb R^3$, the nondegeneracy of the ground states was proved by different methods. J. Wei and M. Winter \cite{WeiWin2009} analyze the Schr\"odinger-Newton equation (\ref{e:sn}) (which is equivalent to (\ref{e:ug})) instead. This approach has an advantage that the Schr\"odinger-Newton equation is local. Therefore, the equation can be reduced to a series of ODE systems. In \cite{Lenzmann2009}, E. Lenzmann finds a method relied on spectral analysis of the linearized operators at the ground states. One of the most important ingredients is the expansion of the nonlocal term by spherical harmonics which is related to the classical multipole expansion of Newton potential in $\mathbb R^3$. Both of the two methods are based on expansion of the equations by spherical harmonics.

To show the nondegeneracy for the case $n\ge 3$, one of key point is to expand the nonlocal term in (\ref{e:ug}) by spherical harmonics.
Let $\mathcal Y_k={\rm span}\{Y_{km}\}_{m\in M_k}$ be the finite dimensional space of real-valued spherical harmonics of degree $k$. Here $M_k=\{1,2,\cdots,\dim \mathcal Y_k\}$, and for $m,m_1,m_2\in M_k$,
\begin{equation*}
-\Delta_{\mathbb S^{n-1}}Y_{km}=k(k+n-2)Y_{km},\quad \langle Y_{km_1},Y_{km_2}\rangle_{L^2(\mathbb S^{n-1})}=\left\{\begin{array}{cc}
                                                                                                                      1, & \mbox{ if }m_1=m_2, \\
                                                                                                                      0, & \mbox{ if }m_1\ne m_2,
                                                                                                                    \end{array}
\right.
\end{equation*}
where $\Delta_{\mathbb S^{n-1}}$ is the Laplacian on $\mathbb S^{n-1}$.

For all $\varphi\in C_0^{\infty}(\mathbb R^n)$, we expanse
\begin{equation}\label{e:npc1}
(I_2\ast\varphi)(x)=\sum_{k=0}^{\infty}\sum_{m\in M_k}\frac{1}{2k+n-2}\left(\int_{0}^{\infty}\frac{(r\wedge \rho)^{k}}{(r\vee \rho)^{k+n-2}}f_{km}(\rho)\rho^{n-1}d\rho\right)Y_{km}(\theta).
\end{equation}
Here $r\wedge \rho=\min\{r,\rho\}$, $r\vee \rho=\max\{r,\rho\}$, $r=|x|$, $\theta=x/|x|$ and $\varphi(x)=\varphi(r,\theta)=\sum_{k=0}^{\infty}\sum_{m\in M_k}f_{km}(r)Y_{km}(\theta)$ is the expansion of $\varphi$ by spherical harmonics.

Using the expansion formula \eqref{e:npc1}, we can prove the nondegeneracy of ground state by the method in \cite{Lenzmann2009} with some modifications.

\subsection{Semiclassical solutions}

As an application of the nondegeneracy of ground state, we construct multiple semiclassical solutions to a class of Hartree type equation with potentials as follows
\begin{equation} \label{e:namain}
-\varepsilon^2\Delta u+u+V(x)u=\varepsilon^{-2}(I_{2}\ast u^2)u\quad \mbox{ in }\mathbb R^n.
\end{equation}
Here $V:\mathbb R^n\to \mathbb R$ an external potential satisfying certain assumptions (see condition ${\bf (V)}$ below).
Let $B_R(0)$ be the ball in $\mathbb R^n$ centered at $0$ with radius $R$, and let
$$C^{2}_{b}(\mathbb R^n):=\{V\in C^{2}(\mathbb R^n)\,|\, D^J V \mbox{ is bounded in }\mathbb R^n \mbox{ for all multi-index } J\mbox{ such that }|J|\le 2\}.$$
We set
\begin{itemize}
  \item[$\bf (V)$:] $V\in C^{2}_{b}(\mathbb R^n)$, $\inf_{x\in \mathbb R^n}(1+V(x))>0$.
\end{itemize}

\begin{theorem}\label{t:multiple}
Let $n=3,4,5$. Assume that $V$ satisfies condition $\bf (V)$ and has a nondegenerate smooth compact critical manifold $M$. Then for $\varepsilon>0$ small, Equation (\ref{e:namain}) has at least $l(M)$ solutions concentrating near points of $M$. Here $l(M)$ denotes the cup length of $M$.
\end{theorem}
Here we say a critical manifold $M$ of $V$ is nondegenerate if, for every $x\in M$, the kernel of $D^2f(x)$ equals to $T_x M$ (\cite{Bo:AM57}).
For the definition of cup length see (\ref{e:cupl}) below.

The existence of semiclassical solutions to Hartree type equations (Schr\"odinger-Newton equations, Choquard equations) with external potentials in $\mathbb R^3$ was obtained by \cite{WeiWin2009} via a Lyapunov-Schmidt reduction argument. To be more precise, they proved that if $V\in C^2(\mathbb R^3)$ has nondegenerate critical points and $\inf_{x\in \mathbb R^3}(1+V(x))>0$, then there exists a family of multibump positive solutions concentrating around those critical points. Related results for other potentials were also investigated. For example, \cite{Secchi2010} showed the existence of semiclassical solutions for vanishing potential under assumption that $\lim_{|x|\to\infty}(1+V(x))|x|^{s}>0$ with $s\in [0,1)$, and \cite{Macri2009,Nolasco2010} verified the case of period potentials.

 An alternative approach to obtain semiclassical solutions is variational penalisation method which does not depend on the nondegeneracy of ground states (see i.e. \cite{BYEON2002,Byeon2003,DELPINO1997,Moroz2009}). Recently, this method has been widely used to construct semiclassical solutions of Choquard type equations with various potentials and nonlinearities. See for example \cite{cingolani2010,Yang2013,Yang2013b,Alves20144133,Alves2014,MorVS2015,Alves2016,Cingolani2015,Alves2015,alves2017singularly,cingolani2017semi} and the references therein.

Our proof of Theorem \ref{t:multiple} is based on a Lyapunov-Schmidt reduction argument introduced by \cite{AmMalSe2001}. Comparing with the local Schr\"odinger equations as in \cite{AmMalSe2001}, more involved analysis is needed to deal the nonlocal term in (\ref{e:namain}). This kind of situations also appears in the equations containing other nonlocal operators such as fractional Laplacian (see i.e. \cite{Davila2014,CZ:CPFSE,Davila2015,Chen2015}). For more details see Section \ref{s:application} below.

The paper is organized as follows. Section \ref{s:preliminaries} contains some preliminaries for further applications. In Section \ref{s:nondegeneracy}, we sketch the proof of the nondegeneracy of ground states (Theorem \ref{t:nondegeneracy}). Section \ref{s:application} is devoted to the applications of the nondegeneracy result to construct multiple semiclassical solutions (Theorem \ref{t:multiple}).

\subsection{Notations}
Finally, we list some notations.
\begin{itemize}
  \item $B_R(x)$ denotes the ball centered at $x$ with radius $R>0$ in $\mathbb R^n$ and $B_R^c(x):=\mathbb R^n\setminus B_R(x)$.
  \item $\mathbb S^{n-1}$ is the unit sphere in $\mathbb R^n$ and $|\mathbb S^{n-1}|$ denotes the $(n-1)$-dimensional volume of $\mathbb S^{n-1}$.
  \item $\langle\cdot,\cdot\rangle$ is the inner product in $H^1(\mathbb R^n)$ and $\|\cdot\|$ denotes the corresonding norm, that is,
  \begin{equation*}
  \|u\|:=\left(\int_{\mathbb R^n}|\nabla u|^2+u^2\right)^{\frac{1}{2}},\quad\forall u\in H^1(\mathbb R^n).
  \end{equation*}
  \item $U$ is the unique positive radial ground state for the Hartree equation (\ref{e:ug}).
  \item $L^2_{\rm rad}(\mathbb R^n):=\{\varphi\in L^2(\mathbb R^n)\,|\,\varphi \mbox{ is radially symmetric}\}$.
\end{itemize}

%% file: Preliminaries.tex
\section{Preliminaries}\label{s:preliminaries}
In this section, we give some basic facts for further applications.

\begin{lemma}\label{l:newton}
For any radial function $f=f(|x|)$ in $L^1(\mathbb R^n,(1+|x|)^{-(n-2)}dx)$, it holds that
\begin{equation}\label{e:nonlocal-potential}
(I_{2}\ast f)(r)=\int_{\mathbb R^n}I_{2}(|x|)f(|x|)-\int_{0}^r K(r,\rho)f(\rho)d\rho,
\end{equation}
where $r=|x|$ and
\begin{equation}\label{e:k}
K(r,\rho)=\frac{\rho}{n-2}\left(1-\frac{\rho^{n-2}}{r^{n-2}}\right),\quad \mbox{for } \rho\le r.
\end{equation}
\end{lemma}
The conclusion of this lemma can be considered as a detailed representation of the Newton's theorem for dimension $n$ (see \cite[Theorem 9.7]{LieLosAnalysis}).
\begin{proof}
It is clear that $I_{2}\ast f$ is radially symmetric. We compute
\begin{eqnarray*}
(I_{2}\ast f)(x)&=&\frac{1}{(n-2)|\mathbb S^{n-1}|}\int_{\mathbb R^n}\frac{f(y)}{|x-y|^{n-2}}dy\\
&=&\frac{1}{(n-2)|\mathbb S^{n-1}|}\int_{B_r(0)}\frac{f(y)}{|x-y|^{n-2}}dy+\frac{1}{(n-2)|\mathbb S^{n-1}|}\int_{B_r^c(0)}\frac{f(y)}{|x-y|^{n-2}}dy,\notag
\end{eqnarray*}
where $r=|x|$.
Since $\frac{1}{|x-y|^{n-2}}$ is harmonic for $x\ne y$, by mean-value formula we have
\begin{eqnarray*}
\int_{B_r}\frac{f(y)}{|x-y|^{n-2}}dy&=&\int_{0}^r\int_{\mathbb S^{n-1}}\frac{f(\rho)\rho^{n-1}}{|x-\rho \omega|^{n-2}}d\omega d\rho\notag\\
&=&|\mathbb S^{n-1}|\int_{0}^rf(\rho)\rho^{n-1}\left(|\mathbb S^{n-1}|^{-1}\int_{\mathbb S^{n-1}}\frac{1}{|x-\rho \omega|^{n-2}}d\omega\right)d\rho \notag\\
&=&\frac{1}{|x|^{n-2}}\int_{B_r}f(y)dy.
\end{eqnarray*}
On the other hand, let
\begin{eqnarray*}
b(x):=\int_{B_r^c}\frac{f(y)}{|x-y|^{n-2}}dy,\quad x\in B_r(0).
\end{eqnarray*}
Then $b$ is harmonic in $B_r(0)$ and radially symmetric. This yields that $b$ is a constant in $\bar B_r(0)$. In particular,
\begin{equation*}
b(x)=b(0)=\int_{B_r^c(0)}\frac{f(y)}{|y|^{n-2}}dy.
\end{equation*}
It follows that
\begin{eqnarray*}
(I_{2}\ast f)(x)&=&\frac{1}{(n-2)|\mathbb S^{n-1}|}\left(\frac{1}{|x|^{n-2}}\int_{B_r}f(y)dy+\int_{B_r^c}\frac{f(y)}{|y|^{n-2}}dy\right)\\
&=&\frac{1}{(n-2)|\mathbb S^{n-1}|}\left(\int_{B_r}\left(\frac{1}{|x|^{n-2}}-\frac{1}{|y|^{n-2}}\right)f(y)dy+\int_{\mathbb R^n}\frac{f(y)}{|y|^{n-2}}dy\right)\notag\\
&=&\frac{1}{(n-2)|\mathbb S^{n-1}|}\left(|\mathbb S^{n-1}|\int_0^r\rho\left(\frac{\rho^{n-2}}{r^{n-2}}-1\right)f(\rho)d\rho+\int_{\mathbb R^n}\frac{f(y)}{|y|^{n-2}}dy\right).\notag
\end{eqnarray*}
Thus (\ref{e:nonlocal-potential}) holds. This completes the proof.
\end{proof}

Note that the ground state $U$ decays exponentially, from the lemma \ref{l:newton} we have that $U\varphi$ belongs to $L^1(\mathbb R^n,(1+|x|)^{-(n-2)}dx)$ for any $\varphi\in L^2_{\rm rad}(\mathbb R^n)$. Thus
\begin{equation*}
(I_{2}\ast |U|^2)(r)=\int_{\mathbb R^n}I_{2}(|x|)|U|^2(|x|)-\int_{0}^r K(r,\rho)|U|^2(\rho)d\rho,
\end{equation*}
and
\begin{equation*}
(I_{2}\ast (U\varphi))(r)=\int_{\mathbb R^n}I_{2}(|x|)(U\varphi)(|x|)-\int_{0}^r K(r,\rho)(U\varphi)(\rho)d\rho.
\end{equation*}

We investigate the decay property of $U'$.
\begin{lemma}\label{l:decayu'}
Assume $\tau$ is an arbitrary number in $(0,1)$. Then
\begin{equation}\label{e:u'decay}
|U'|\le Ce^{-\tau r},
\end{equation}
where $r=|x|$ and $C$ is a positive constant depending on $\tau$. As a consequence, we have
\begin{equation}\label{e:pu'decay}
|\partial_{x_i}U|\le Ce^{-\tau r},\quad i=1,2,\cdots,n.
\end{equation}
\end{lemma}
\begin{proof}
Let $U>0$ be the unique positive radial ground state. Then $U'\le 0$.
Taking derivative with respect to $r$ on Equation (\ref{e:ug}), we obtain that
\begin{eqnarray}\label{e:uradial-d1}
-U'''-\frac{n-1}{r}U''+\left(\frac{n-1}{r^2}+1\right)U'=(I_{2}\ast U^2)'U+(I_{2}\ast U^2)U'.
\end{eqnarray}
Again by Lemma \ref{l:newton},
\begin{eqnarray}\label{e:iad1}
(I_{2}\ast U^2)'(r)&=&-K(r,r)U^2(r)-\int_{0}^r\partial_rK(r,\rho)U^2(\rho)d\rho\notag\\
&=&-\frac{1}{r^{n-1}}\int_{0}^r \rho^{n-1} U^2(\rho) d\rho.
\end{eqnarray}
From (\ref{e:uradial-d1}) and the decay property of $U$, we have that
\begin{eqnarray*}
|(I_2\ast U^2)(x)|\le \frac{C}{r^{n-2}}\quad\mbox{ and }\quad\left|\int_{0}^r \rho^{n-1} U^2(\rho) d\rho\right|\le C.
\end{eqnarray*}
Therefore, for any $\epsilon\in (0,1)$ there exists $R>0$ sufficiently large such that for all $x\in B_R^c(0)$
\begin{eqnarray*}
-\Delta (- U')+(1-\epsilon)(-U')\le C_1e^{-\tau|x|},
\end{eqnarray*}
where $C_1$ is a positive constant and $0<\tau\le 1-\epsilon$. It follows that
\begin{eqnarray*}
-\Delta (- U')+(1-\epsilon)(-U')\le -\Delta(C_2e^{-\tau|x|})+(1-\epsilon)C_2e^{-\tau|x|},\quad \mbox{for all }x\in B_R^c(0),
\end{eqnarray*}
where $C$ is some positive constant. That is,
\begin{eqnarray*}
-\Delta (-U'-Ce^{-\tau|x|})+(1-\epsilon)(-U'-Ce^{-\tau|x|})\le 0,\quad \mbox{for all }x\in B_R^c(0).
\end{eqnarray*}
Using the comparison principle in \cite[Proposition 6.3]{MorVSJFA2013}, we have that
\begin{eqnarray*}
- U'-Ce^{-\tau|x|}\le C_3|x|^{-\frac{N-1}{2}}e^{-(1-\epsilon)^{\frac{1}{2}}|x|},\quad \mbox{for all }x\in B_R^c(0).
\end{eqnarray*}
Thus we obtain
\begin{eqnarray*}
- U'\le Ce^{-\tau|x|},\quad \mbox{for all }x\in B_R^c(0),
\end{eqnarray*}
where $\tau\in (0,1-\varepsilon]$. Thus (\ref{e:u'decay}) follows since $\epsilon$ is a arbitrary number in $(0,1)$.

A direct computation yields that
\begin{equation*}
\partial_{x_j}U(r)=U'(r)\frac{x_j}{r}.
\end{equation*}
Then (\ref{e:pu'decay}) holds. This completes the proof.
\end{proof}

\begin{lemma}\label{l:l2u+ru'}
It holds that
\begin{equation}\label{e:l2u+ru'}
L(2U+rU')=-2U.
\end{equation}
\end{lemma}
\begin{proof}
Firstly, by (\ref{e:ug}), we have
\begin{eqnarray}\label{e:lu}
LU=-\Delta U+U-(I_{2}\ast U^2)U-2(I_{2}\ast U^2)U=-2(I_{2}\ast U^2)U.
\end{eqnarray}

Secondly, we compute $L(rU')$. From (\ref{e:linearised}),
\begin{eqnarray}\label{e:lru'l}
L(rU')=-\Delta (rU')+rU'-(I_{2}\ast U^2)(rU')-2(I_{2}\ast (r U U'))U.
\end{eqnarray}
Using the polar coordinates representation of Laplacian, we have
\begin{equation}\label{e:lru'}
-\Delta (rU')=-r\left(U'''+\frac{n-1}{r}U''\right)-2U''-\frac{n-1}{r}U'.
\end{equation}
By Lemma \ref{l:newton},
\begin{eqnarray}\label{e:iup-1ru}
2(I_{2}\ast (rUU'))(x)&=&2\int_{\mathbb R^n}|x|I_{2}(x)(UU')(x)-2\int_0^r K(r,\rho)\rho U(\rho)U'(\rho)d\rho.
\end{eqnarray}
Integration by parts and (\ref{e:k}) yield that
\begin{eqnarray}\label{e:kinbyp}
2\int_0^r K(r,\rho)\rho U(\rho)U'(\rho)d\rho
=-\frac{1}{n-2}\int_0^r\rho(2-n\frac{\rho^{n-2}}{r^{n-2}})U^2(\rho)d\rho\notag
\end{eqnarray}
Then from (\ref{e:uradial-d1}) and (\ref{e:iad1}), we get
\begin{multline}\label{e:u'''}
-U'''-\frac{n-1}{r}U''=-\left(\frac{n-1}{r^2}+1\right)U'\\
-\frac{1}{r^{n-1}}\left(\int_{0}^r \rho^{n-1} U^2(\rho) d\rho\right)U+(I_{2}\ast U^2)U'.
\end{multline}
Putting (\ref{e:ug}), (\ref{e:lru'}) and (\ref{e:u'''}) together, we find that
\begin{eqnarray}\label{e:lru'pc}
-\Delta (rU')&=&-2U''-\left(\frac{2(n-1)}{r}+r\right)U'\notag\\
&&-\frac{1}{r^{n-2}}\left(\int_{0}^r \rho^{n-1} U^2(\rho) d\rho\right)U
+(I_{2}\ast U^2)rU'\notag\\
&=&-2U+2(I_{2}\ast U^2)U-rU'\notag\\
&&-\frac{1}{r^{n-2}}\left(\int_{0}^r \rho^{n-1} U^2(\rho) d\rho\right)U
+(I_{2}\ast U^2)rU'.
\end{eqnarray}
Moreover, integrating by parts, it holds that
\begin{eqnarray}\label{e:iuu'}
2\int_{\mathbb R^n}|x|I_{2}(x)(UU')(x)&=&\frac{1}{n-2}\int_{0}^{\infty}\rho^2U(\rho)U'(\rho)d\rho\\
&=&-\frac{2}{n-2}\int_{0}^{\infty}\rho U^2(\rho)d\rho=-2\int_{\mathbb R^n}I_{2}(x)U^2(x)dx.\notag
\end{eqnarray}
Then from (\ref{e:lru'l}), (\ref{e:lru'pc}), (\ref{e:iup-1ru}), (\ref{e:kinbyp}) and (\ref{e:iuu'}), it holds that
\begin{equation}\label{e:lru'1}
L(rU')=-2U+4(I_{2}\ast U^2)U.
\end{equation}
Hence (\ref{e:l2u+ru'}) follows from (\ref{e:lu}) and (\ref{e:lru'1}). This completes the proof.
\end{proof}

%% file: Nondegeneracy.tex
\section{Nondegeneracy}\label{s:nondegeneracy}

In this section, we prove Theorem \ref{t:nondegeneracy}.

To show the nondegeneracy for the case $n\ge 3$, we need to expand the nonlocal term in (\ref{e:ug}) by spherical harmonics.
Let $\mathcal Y_k={\rm span}\{Y_{km}\}_{m\in M_k}$ be the finite dimensional space of real-valued spherical harmonics of degree $k$. Here $M_k=\{1,2,\cdots,\dim \mathcal Y_k\}$, and for $m,m_1,m_2\in M_k$,
\begin{equation*}
-\Delta_{\mathbb S^{n-1}}Y_{km}=k(k+n-2)Y_{km},\quad \langle Y_{km_1},Y_{km_2}\rangle_{L^2(\mathbb S^{n-1})}=\left\{\begin{array}{cc}
                                                                                                                      1, & \mbox{ if }m_1=m_2, \\
                                                                                                                      0, & \mbox{ if }m_1\ne m_2,
                                                                                                                    \end{array}
\right.
\end{equation*}
where $\Delta_{\mathbb S^{n-1}}$ is the Laplacian on $\mathbb S^{n-1}$.
\begin{prop}\label{p:expansion}
For all $\varphi\in C^{\infty}_0(\mathbb R^n)$, we expanse
\begin{equation}\label{e:npc}
(I_2\ast\varphi)(x)=\sum_{k=0}^{\infty}\sum_{m\in M_k}\frac{1}{2k+n-2}\left(\int_{0}^{\infty}\frac{(r\wedge \rho)^{k}}{(r\vee \rho)^{k+n-2}}f_{km}(\rho)\rho^{n-1}d\rho\right)Y_{km}(\theta).
\end{equation}
Here $r\wedge \rho=\min\{r,\rho\}$, $r\vee \rho=\max\{r,\rho\}$, $r=|x|$, $\theta=x/|x|$ and $\varphi(x)=\varphi(r,\theta)=\sum_{k=0}^{\infty}\sum_{m\in M_k}f_{km}(r)Y_{km}(\theta)$ is the expansion of $\varphi$ by spherical harmonics.
\end{prop}
\begin{proof}
From \cite[Example 3.5.12]{simon2015harmonic}, we have the following multipole expansion
\begin{eqnarray}
\frac{1}{|x-y|^{n-2}}=(n-2)|\mathbb S^{n-1}|\sum_{k=0}^{\infty}\sum_{m\in M_k}\left(\frac{1}{2k+n-2}\right)\frac{(r\wedge \rho)^{k}}{(r\vee \rho)^{k+n-2}}Y_{km}(\theta)Y_{km}(\omega),
\end{eqnarray}
where $\theta=x/|x|$ and $\omega=y/|y|$. Then direct computations yield the Equation \eqref{e:npc}.
\end{proof}

Using spherical harmonics, we have
\begin{equation*}
L^2(\mathbb R^n)=\bigoplus_{k\ge 0}\mathcal L^2_k,
\end{equation*}
where
$\mathcal L^2_k=L^2(\mathbb R_+,r^{n-1}dr)\otimes\mathcal Y_k.$
For simplicity, set $\mathcal Y_{km}={\rm span}\{Y_{km}\}$ and $\mathcal L^2_{km}=L^2(\mathbb R_+,r^{n-1}dr)\otimes\mathcal Y_{km}$.
Any $\varphi\in L^2(\mathbb R^n)$ is of form
\begin{equation*}
\varphi(x)=\sum_{k=0}^{\infty}\sum_{m\in M_k} f_{km}(r)Y_{km}(\theta).
\end{equation*}
Here $x=(r,\theta)$ with $r=|x|$ and $\theta=x/|x|$. Moreover, set $\mathcal C_k=C^{\infty}_0(\mathbb R_+)\otimes\mathcal Y_k\subset \mathcal L^2_k$ and $\mathcal C_{km}=C^{\infty}_0(\mathbb R_+)\otimes\mathcal Y_{km}\subset\mathcal L^2_{km}$.

\begin{lemma}\label{l:linvariant}
For any function $\varphi_{km}(x)=f_{km}(r)Y_{km}(\theta)\in \mathcal C_{km}$, $k\in \mathbb N_+$, $m\in M_k$, it holds that $L(\varphi)\in \mathcal L^2_{km}$.
\end{lemma}
\begin{proof}
1. We recall that in polar coordinates,
\begin{equation}\label{e:lpcl}
-\Delta=-\partial^2_r-\frac{n-1}{r}\partial_r-\frac{1}{r^2}\Delta_{\mathbb S^{n-1}}.
\end{equation}
Then
\begin{eqnarray}\label{e:lpc1}
-\Delta (f_{km}(r)Y_{km}(\theta))= \left(-f_{km}''(r)-\frac{n-1}{r}f_{km}'(r)+\frac{k(k+n-2)}{r^2}f_{km}(r)\right)Y_{km}(\theta).
\end{eqnarray}
which belongs to $\mathcal L^2_{km}$. It is clearly that $(1-(I_2\ast U^2))f_{km}(r)Y_{km}(\theta)\in \mathcal L^2_{km}$.

2. We now focus on the nonlocal term. For simplicity, set $P(\varphi)=-U(I_2\ast (U\varphi))$.
From Proposition \ref{p:expansion}, we have that 
\begin{eqnarray}\label{e:expan2}
P(\varphi)=-\frac{U(r)}{2k+n-2}\left(\int_{0}^{\infty}\frac{(r\wedge \rho)^{k}}{(r\vee \rho)^{k+n-2}}U(\rho)f_{km}(\rho)\rho^{n-1}d\rho\right)Y_{km}(\theta).
\end{eqnarray}
Since $U$ is exponential decaying, it holds that $P(\varphi)\in \mathcal L^2_{km}$.
This complete the proof.
\end{proof}

Thanks to the proof of Lemma \ref{l:linvariant}, we can define $P_{km}$ (resp. $L_{km}$) as the restriction of $P$ (resp. $L$) on $\mathcal C_{km}$. Furthermore, we set
\begin{equation}\label{e:dvp}
\mbox{$J(\varphi):=-(I_2\ast U^2)(\varphi)$, and $P_k=\sum_{m\in M_k}P_{km}$.}
\end{equation}

For any nontrivial $f\in \mathcal L^2_{km}$, from \eqref{e:expan2}, it holds that $P_{km}f\in \mathcal L^2_{km}$.
By \cite[Theorem 9.8]{LieLosAnalysis} we obtain that
\begin{eqnarray}\label{e:l21}
\langle P_{km}f,f\rangle_{\mathcal L^2_{km}}
=\int_{\mathbb R^n}\int_{\mathbb R^n}\frac{U(x)f(x)U(y)f(y)}{|x-y|^{n-2}}dxdy
>0.
\end{eqnarray}
Since $U$ is bounded and exponentially decaying, the Hardy-Littlewood-Sobolev inequality (see, for example, \cite[Theorem 4.3]{LieLosAnalysis}) yields that for all $h\in L^2(\mathbb R^n)$,
\begin{multline}\label{e:l22}
|\langle P_{km}f,h\rangle_{L^2(\mathbb R^n)}|
\le\int_{\mathbb R^n}\int_{\mathbb R^n}\left|\frac{U(x)h(x)U(y)f(y)}{|x-y|^{n-2}}\right|dxdy
\le C_1\|Uh\|_{L^2(\mathbb R^n)}\|Uf\|_{L^{\frac{2n}{n+4}}(\mathbb R^n)}\\
\le C_2\|h\|_{L^2(\mathbb R^n)}\|f\|_{L^2(\mathbb R^n)}\|U\|_{L^{\frac{n}{2}}(\mathbb R^n)}\le C_3\|h\|_{L^2(\mathbb R^n)}\|f\|_{L^2(\mathbb R^n)}.
\end{multline}

We next investigate the restriction operator of $L$ on $L^2_{\rm rad}(\mathbb R^n)$.
\begin{prop}\label{t:nondeneracy-radial}
Let $L$ be the linearised operator given by (\ref{e:linearised}). Then we have
\begin{equation*}
{\rm ker}\, (L|_{L^2_{\rm rad}(\mathbb R^n)})=\{0\}.
\end{equation*}
\end{prop}
We split $L$ into two operators. To be more precise, we define
\begin{equation}\label{e:m2}
\Upsilon\varphi=-\Delta \varphi+\varphi-(I_{2}\ast U^2)\varphi+W\varphi,\quad \forall \varphi\in H^2(\mathbb R^n).
\end{equation}
where
\begin{equation*}
(W\varphi)(r)=2U(r)\int_0^r K(r,\rho)U(\rho)\varphi(\rho)d\rho,\quad \forall \varphi\in H^2(\mathbb R^n).
\end{equation*}
Then
\begin{equation*}
L\varphi=\Upsilon\varphi-\Xi\varphi,\quad \forall \varphi\in H^2(\mathbb R^n),
\end{equation*}
where
\begin{equation}\label{e:dfp}
\Xi\varphi=2\left(\int_{\mathbb R^n}\frac{U\varphi}{|x|^{n-2}}\right)U,\quad \forall \varphi\in H^2(\mathbb R^n).
\end{equation}

\begin{proof}[Proof of Proposition \ref{t:nondeneracy-radial}]
The proof is a modification of \cite[Proof of Proposition 2]{Lenzmann2009}. We outline the proof here for completeness.

1. We first prove the following result for operator $\Upsilon$.

\emph{Claim: If $v=v(r)$ is a radial solution to $\Upsilon v=0$ with $v(0)\ne 0$ and $v'(0)=0$, then $v(r)$ has no sign change and for $\varrho\in (0,1)$, there exist constant $C>0$ and $R>0$ such that
\begin{equation}\label{e:vdecay}
|v(r)|\ge C e^{\varrho r}, \quad \mbox{for } r\ge R.
\end{equation}}

In fact, without loss of generality, assume that $v(0)>U(0)>0$. In the polar coordinates, $v$ satisfies
\begin{equation}\label{e:vr}
v''(r)+\frac{n-1}{r}v'(r)=G(r)v(r)+(Wv)(r),
\end{equation}
where
\begin{equation*}
G(r)=1-(I_{2}\ast |U|^2)(r).
\end{equation*}
Note that $U$ satisfies
\begin{equation}\label{e:ur}
U''(r)+\frac{n-1}{r}U'(r)=G(r)U(r).
\end{equation}
Combining (\ref{e:vr}) with (\ref{e:ur}), we get that
\begin{equation*}
[r^{n-1}(Uv'-U'v)]'(r)=r^{n-1}U(r)(Wv)(r).
\end{equation*}
It follows that
\begin{equation}\label{e:v/u}
r^{n-1}\left(\frac{v}{U}\right)'(r)=\frac{1}{U^2(r)}\int_{0}^r\rho^{n-1}U(\rho)(Wv)(\rho)d\rho.
\end{equation}
By the assumption $v(0)>U(0)$, it holds that $v(r)>U(r)$ for $r>0$ sufficiently small. Assume that $r_0>0$ is the smallest number such that $v(r_0)=U(r_0)$. Then $\left(v/U\right)'(r_0)\le 0$. However, the right side of (\ref{e:v/u}) at $r=r_0$ is strictly positive since $v(r)>U(r)>0$ in $[0,r_0)$ and $K(r,\rho)>0$ for $0<\rho<r$. This is impossible. Hence we have
\begin{equation}\label{e:vgeu}
v(r)>U(r)\quad \mbox{for all } r\in(0,\infty).
\end{equation}
As a consequence, $v(r)$ is strictly positive.
From (\ref{e:vgeu}) and (\ref{e:v/u}), we obtain that
\begin{eqnarray}\label{e:v/udecay}
r^{n-1}\left(\frac{v}{U}\right)'(r)\ge 2U^{-2}(r)\int_0^r\rho^{n-1}U^{2}(\rho)\int_{0}^{\rho} K(\rho,t)U^2(t)dtd\rho.
\end{eqnarray}
Since $U$ is positive, it follows that for sufficiently large $r>0$, there exists a constant $C_1$
\begin{equation}\label{e:kint}
0<C_1<\int_0^r\rho^{n-1}u^{p}(\rho)\int_{0}^{\rho} K(\rho,t)u^p(t)dtd\rho<\infty.
\end{equation}
Let $\epsilon>0$ be a number sufficiently small. From (\ref{e:Uasymptotics}), we find that when $r$ is sufficiently large, there exists a constant  $c>0$ such that
\begin{equation}\label{e:udecay}
U(r)<ce^{-(1-\epsilon) r}.
\end{equation}
Thus from (\ref{e:v/udecay}), (\ref{e:udecay}) and (\ref{e:kint}), there exists $R_1>0$
\begin{equation*}
r^{n-1}\left(\frac{v}{U}\right)'(r)\ge C_2 e^{2(1-\epsilon) r},\quad \mbox{ for all }r>R_1.
\end{equation*}
Therefore, integrating both side of the inequality above, we obtain (\ref{e:vdecay}). Hence the claim holds.

2. Based on Step 1, we show the proposition by contradiction. Assume $\varphi\in L^{2}_{\rm rad}(\mathbb R^n)$ is a nontrivial solution to $L\varphi=0$. This means
\begin{equation*}
\Upsilon\varphi=\Xi\varphi:=2\sigma U,
\end{equation*}
where $\sigma=\int_{\mathbb R^n}\frac{U\varphi}{|x|^{n-2}}$.
Consider
\begin{equation}\label{e:m-nonhomogeneous}
\Upsilon w=\Xi\varphi.
\end{equation}
Then
\begin{equation*}
\varphi=v+w,
\end{equation*}
where $w$ is any solution to (\ref{e:m-nonhomogeneous}) and $v$ is a solution to $\Upsilon v=0$.
We construct a particular $w$.
Let $\eta=2U+rU'$. Clearly, $\eta$ belongs to $L^2_{\rm rad}(\mathbb R^n)$. By (\ref{e:dfp}) and Lemma \ref{l:l2u+ru'}, we have
\begin{equation*}
\Upsilon\eta=\Xi\eta-2U=2(\tau-1)U,
\end{equation*}
where $\tau=\int_{\mathbb R^n}\frac{U\eta}{|x|^{n-2}}$. Since $\eta$ is radial symmetric, it holds that $v\ne \eta$. Hence, by the result in Step 1, we obtain that $\tau\ne 1$. It follows that
\begin{equation*}
w=\frac{\sigma}{\tau-1}\eta\in L^2_{\rm rad}(\mathbb R^n)
\end{equation*}
is a solution to (\ref{e:m-nonhomogeneous}).

If $v\equiv0$, then $\varphi=w$. Note that $\Xi\varphi\ne 0$. Thus $Lw=-(2\sigma/(\tau-1))U\ne 0$. It contradicts that $L\varphi=0$. Therefore $v\ne 0$. Since $v$ is radially symmetric and of $C^2$ class in $\mathbb R^n$, we have $v'(0)=0$. Assume that $v(0)\ne 0$. Then Step 1 yields that $v\notin L^2_{\rm rad}(\mathbb R^n)$. This contradicts $v=\varphi-w\in L^2_{\rm rad}(\mathbb R^n)$. Hence $v(0)=0$. By uniqueness of radial solutions to $\Upsilon v=0$, it holds that $v\equiv 0$. It is a contradiction. This completes the proof.
\end{proof}

With the preparations above, we are now in place to compute the kernel of $L$. This procedure is similar with that in the three dimension case (\cite{Lenzmann2009}). We shall only give a sketch of the proof.

\begin{proof}[Proof of Theorem \ref{t:nondegeneracy}]
We shall prove the theorem by 4 steps as follows.

1. The restriction of $-\Delta$ on $\mathcal C_k$ is given by
\begin{equation*}
-\Delta_k=-\partial^2_r-\frac{n-1}{r}\partial_r+\frac{k(k+n-2)}{r^2}.
\end{equation*}
From \cite[Theorem X. 10 and Example 4]{ReedSimonII}, we obtain that for $k\ge 1$, $-\Delta_k$ is essentially self-adjoint on $C_0^{\infty}(\mathbb R_+)\subset L^2(\mathbb R_+,r^{n-1}dr)$. Moreover, \eqref{e:l21} and \eqref{e:l22} tells us that $J$ and $P_k$ is bounded and self-adjoint in $\mathcal L^2_k$. Then Kato-Rellich theorem yields that $L_k=-\Delta_k+J+P_k$ is essentially self-adjoint on $C_0^{\infty}(\mathbb R_+)\subset L^2(\mathbb R_+,r^{n-1}dr)$ (with an abuse of the notations, $J,P_k$ denote the corresponding operators on $L^2(\mathbb R_+,r^{n-1}dr)$ for simplicity).

2. For $\lambda>0$ sufficiently large, it holds that
\begin{equation*}
(L_k+\lambda)^{-1}=(-\Delta_k+\lambda+J+P_k)^{-1}=(-\Delta_k+\lambda)^{-1}(I+(J+P_k)(-\Delta_k+\lambda)^{-1})^{-1}.
\end{equation*}
We recall that for $\lambda>0$, the resolvent $(-\Delta_k+\lambda)^{-1}$ is positivity improving on $L^2(\mathbb R_+,r^{n-1}dr)$ (\cite[Equation (7.15)]{Lenzmann2009}).
For simplicity, we set $A_k=J+P_k$.
By the definition of $J$ and $P_k$ (see (\ref{e:dvp})), we have that $-A_k$ is positivity improving on $L^2(\mathbb R_+,r^{n-1}dr)$.
Recalling that $A_k$ is bounded, we have that for $\lambda$ sufficiently large,
\begin{equation*}
\|A_k(-\Delta_k+\lambda)^{-1}\|_{L^2(\mathbb R_+,r^{n-1}dr)} <1.
\end{equation*}
By Neumann expansion, we get
\begin{equation}\label{e:llrne}
(L_k+\lambda)^{-1}=(-\Delta_k+\lambda)^{-1}\sum_{j=1}^{\infty}(-A_k(-\Delta_k+\lambda)^{-1})^j.
\end{equation}
It follows that $(L_k+\lambda)^{-1}$ is positivity improving as well.

3. Let $\lambda_{k,0}=\inf \sigma(L_k)$ be the lowest eigenvalue. Note that $(L_k+\lambda)^{-1}$ is bounded and self-adjoint, and its largest eigenvalue $\mu_0=\sup\sigma((L_k+\lambda)^{-1})=(\lambda_{k,0}+\lambda)^{-1}$. Moreover, the corresponding eigenspaces of $L_k$ and $(L_k+\lambda)^{-1}$ coincide. From \cite[Theorem XIII. 43]{ReedSimonIV}, it holds that $\lambda_0$ is simple and the corresponding eigenfunction is strictly positive on $\mathbb R_+$ (Perron-Frobenius property) since $(L_k+\lambda)^{-1}$ is positivity improving. Therefore, $L_{k}$ has Perron-Frobenius property as well.

4. We finally compute ${\rm ker}L$. It is clear that $\partial_{x_1} U,\partial_{x_2} U,\cdots,\partial_{x_n} U$ belong to ${\rm ker}\,L$. Note that
\begin{equation*}
\partial_{x_j}U(r)=U'(r)\frac{x_j}{r}.
\end{equation*}
Recalling that $x_j/r$ belongs to $\mathcal Y_1$, we have that
\begin{equation*}
L_1 U'=0.
\end{equation*}
Note that $U'\le 0$. Then by the result of Step 3, there exists a constant $c<0$ such that
\begin{equation*}
U'(r)=c\phi_{0,1}(r),
\end{equation*}
where $\phi_{0,1}$ is the strictly positive ground state of $L_1$ which corresponds the lowest eigenvalue $\lambda_{1,0}=0$. Hence ${\rm ker}L\cap \mathcal L^2_1={\rm span}\{\partial_{x_1}U, \partial_{x_2}U,\cdots,\partial_{x_n}U\}$.

 We now prove that
$L_k>0, \mbox{ for }k\ge 2$.
In particular, ${\rm ker}L\cap \mathcal L^2_k=0$.

Indeed, for any fixed $k\ge 2$, Step 3 and \cite[Remarks of Lemma 7]{Lenzmann2009} yield that $\lambda_{k,0}<1$ is simple and its corresponding eigenfunction $\phi_{k,0}$ is strictly positive. Therefore, by Proposition \ref{p:expansion},
\begin{equation*}
\lambda_{k,0}=\langle \phi_{k,0},L_k\phi_{k,0}\rangle=\langle \phi_{k,0},L_1\phi_{k,0}\rangle+W_k,
\end{equation*}
where
\begin{eqnarray*}
&&W_k:=\langle \phi_{k,0},(L_k-L_1)\phi_{k,0}\rangle\\
&=&\int_{0}^{\infty}\frac{(k(k+n-2)-(n-1))}{r^2}\phi_{k,0}^2r^{n-1}dr+\notag\\
&&2\int_0^{\infty}\int_0^{\infty}U(r)\phi_{k,0}(r)
\left(\frac{1}{n}\frac{r\wedge \rho}{(r\vee \rho)^{n-2}}-\frac{1}{2k+n-2}\frac{(r\wedge \rho)^k}{(r\vee \rho)^{k+n-2}}\right)\\
&&\quad\quad\quad\quad \quad\quad\quad\quad\quad\quad\quad\quad\quad\quad\quad\quad\quad\quad\quad\quad\quad U(\rho)\phi_{k,0}(\rho)r^{n-1}\rho^{n-1}d\rho dr.\notag
\end{eqnarray*}
Here $r\wedge \rho=\min\{r,\rho\}$ and $r\vee \rho=\max\{r,\rho\}$. Note that $\frac{1}{n}\frac{(r\wedge \rho)}{(r\vee \rho)^{n-2}}-\frac{1}{2k+n-2}\frac{(r\wedge \rho)^k}{(r\vee \rho)^{k+n-2}}>0$ for $k\ge 2$. Recalling that $U$ and $\phi_{k,0}$ are strictly positive, we obtain that
\begin{equation*}
W_k>0,\quad\mbox{ for all }k\ge 2.
\end{equation*}
Since $\langle \phi_{k,0},L_1\phi_{k,0}\rangle \ge \lambda_{1,0}=0$, we finally have
\begin{equation*}
\lambda_{k,0}\ge W_k>0, \quad\mbox{ for all }k\ge 2.
\end{equation*}
This completes our proof.
\end{proof}

%% file: Applications.tex
\section{Proof of Theorem \ref{t:multiple}}\label{s:application}
In this section, we construct multiple semi-classical solutions to (\ref{e:namain}) as an application of the nondegeneracy of ground states.

\subsection{Functional setting}
By a change of variables $x\to \varepsilon x$, Equation (\ref{e:namain}) becomes
\begin{equation} \label{e:namain1}
-\Delta v+v+V(\varepsilon x)v=(I_{2}\ast v^2)v\quad \mbox{ in }\mathbb R^n.
\end{equation}
Solutions of (\ref{e:namain1}) correspond to the critical points of
\begin{eqnarray*}
f_{\varepsilon}(v)&=&\frac{1}{2}\|v\|^2-\frac{1}{4}\int_{\mathbb R^n}(I_2\ast v^2) v^2dx+\frac{1}{2}\int_{\mathbb R^n}V(\varepsilon x)v^2dx\\
&:=& f_0(v)+\frac{1}{2}\int_{\mathbb R^n}V(\varepsilon x)v^2dx,\quad\quad v\in H^1(\mathbb R^n).\notag
\end{eqnarray*}
For $n=3,4,5$, we have that $f_{\varepsilon}\in C^2(H^1(\mathbb R^n))$. We shall verify that there exist solutions of (\ref{e:namain1}) near a solution of
\begin{equation} \label{e:namain2}
-\Delta v+v+V(\varepsilon \xi)v=(I_{2}\ast v^2)v\quad \mbox{ in }\mathbb R^n,
\end{equation}
with some appropriate $\xi$ in $\mathbb R^n$. The solutions of (\ref{e:namain2}) are critical points of
\begin{equation*}
F_{\varepsilon,\xi}(v)=f_0(v)+\frac{1}{2}V(\varepsilon \xi)\int_{\mathbb R^n}v^2dx,\quad\quad v\in H^1(\mathbb R^n).
\end{equation*}
Then we have
\begin{equation}\label{e:fvFvx}
f_{\varepsilon}(v)=F_{\varepsilon,\xi}(v)+\frac{1}{2}\int_{\mathbb R^n}(V(\varepsilon x)-V(\varepsilon\xi))v^2dx.
\end{equation}

Let $\beta=\beta(\varepsilon \xi)=(1+V(\varepsilon \xi))^{\frac{1}{2}}$ and $\alpha=\alpha(\varepsilon \xi)=(1+V(\varepsilon \xi))=(\beta(\varepsilon\xi))^2$. Then $\alpha U(\beta x)$ is a solution to (\ref{e:namain2}). Let
\begin{equation}\label{e:dfzx}
z_{\varepsilon,\xi}:=\alpha(\varepsilon\xi) U(\beta(\varepsilon\xi) x)
\end{equation}
and $Z_{\varepsilon}:=\{z_{\varepsilon,\xi}(x-\xi)\,|\,\xi\in\mathbb R^n\}.$
For simplicity, we set $z_{\xi}=z_{\varepsilon,\xi}(x-\xi)$.

\begin{lemma}\label{l:pxpx}
For all $\xi\in\mathbb R^n$,
\begin{equation}\label{e:pzv}
\partial_{\xi_i}[z_{\varepsilon,\xi}(x-\xi)]=-\partial_{x_i}[z_{\varepsilon,\xi}(x-\xi)]+O(\varepsilon).
\end{equation}
\end{lemma}
\begin{proof}
A direct calculation yields that
\begin{eqnarray*}
&&\partial_{\xi_i}[z_{\varepsilon,\xi}(x-\xi)]\\
&=&\partial_{\xi_i}[\alpha(\varepsilon\xi) U(\beta(\varepsilon\xi)( x-\xi)]\notag\\
&=&\varepsilon [\partial_{\xi_i}\alpha](\varepsilon\xi)U(\beta(\varepsilon\xi)
(x-\xi))+\varepsilon\alpha(\varepsilon\xi)[\partial_{\xi_i}\beta](\varepsilon\xi)[\nabla U](\beta(\varepsilon\xi)
(x-\xi))\cdot (x-\xi)\notag\\
&&-\alpha(\varepsilon\xi)\beta(\varepsilon\xi)[\partial_{x_i}U](\beta(\varepsilon\xi)(x-\xi)).\notag
\end{eqnarray*}
By the definition of $\alpha,\,\beta$ and the decay property of $U$ and $U'$ (Lemma \ref{l:decayu'}), we have estimate (\ref{e:pzv}).
\end{proof}
The Fr{\'e}chet derivative of $f_{\varepsilon}$ at $z_{\xi}$ is small. That is,
\begin{lemma}\label{l:dfvz}
For all $\xi\in \mathbb R^n$ and all $\varepsilon>0$ small,
\begin{equation*}
\|D f_{\varepsilon}(z_{\xi})\|\le C(\varepsilon|\nabla V(\varepsilon \xi)|+\varepsilon^2),
\end{equation*}
where $C$ is a constant independent on $\xi$ and $\varepsilon$.
\end{lemma}
\begin{proof}
Since $z_{\xi}$ is a critical point of $F_{\varepsilon,\xi}$, using (\ref{e:fvFvx}) we obtain that
\begin{eqnarray*}
Df_{\varepsilon}(z_{\xi})(w)=\int_{\mathbb R^n}(V(\varepsilon x)-V(\varepsilon\xi))z_{\xi}wdx,\quad \forall w\in H^1(\mathbb R^n).
\end{eqnarray*}
Then the H\"older inequality yields that
\begin{equation*}
|Df_{\varepsilon}(z_{\xi})(w)|^2\le \left(\int_{\mathbb R^n }|V(\varepsilon x)-V(\varepsilon\xi)|^2z_{\xi}^2dx\right)\|w\|^2_{L^2}.
\end{equation*}
It follows that
\begin{equation}\label{e:dfne}
\|Df_{\varepsilon}(z_{\xi})\|\le \left(\int_{\mathbb R^n }|V(\varepsilon x)-V(\varepsilon\xi)|^2z_{\xi}^2dx\right)^{\frac{1}{2}}.
\end{equation}
Therefore, since $D^2 V$ is bounded and $U$ decays exponentially, we have
\begin{eqnarray}\label{e:nite}
&&\int_{\mathbb R^n }|V(\varepsilon x)-V(\varepsilon\xi)|^2z_{\xi}^2dx\\
&\le& C\left(\varepsilon^2|\nabla V(\varepsilon\xi)|^2\int_{\mathbb R^n}|x-\xi|^2z_{\xi}^2dx+\varepsilon^4\int_{\mathbb R^n}|x-\xi|^4z_{\xi}^2dx\right)\notag\\
&\le &C(\varepsilon^2|\nabla V(\varepsilon\xi)|^2+\varepsilon^4).\notag
\end{eqnarray}
Hence by (\ref{e:dfne}) and (\ref{e:nite}), we get the conclusion. This completes the proof.
\end{proof}

\begin{lemma}\label{l:hls}
Assume that $n=3,4,5$. Then it holds that for all $\varphi_j\in H^1(\mathbb R^n)$ ($j=1,2,3,4$),
\begin{equation}\label{e:hls}
\left|\int_{\mathbb R^n}\int_{\mathbb R^n}\frac{\varphi_1(x)\varphi_2(x)\varphi_3(y)\varphi_4(y)}{|x-y|^{n-2}}dxdy\right|\le C\|\varphi_1\|\|\varphi_2\|\|\varphi_3\|\|\varphi_4\|,
\end{equation}
where $C$ is a constant independent on $\varphi_j$ ($j=1,2,3,4$).
\end{lemma}
\begin{proof}
By the Hardy-Littlewood-Sobolev inequality and H\"older inequality, we have that
\begin{eqnarray*}
\left|\int_{\mathbb R^n}\int_{\mathbb R^n}\frac{\varphi_1(x)\varphi_2(x)\varphi_3(y)\varphi_4(y)}{|x-y|^{n-2}}dxdy\right|&\le & \|\varphi_1\varphi_2\|_{L^{\frac{2n}{n+2}}}\|\varphi_3\varphi_4\|_{L^{\frac{2n}{n+2}}}\\
&\le& \|\varphi_1\|_{L^{\frac{4n}{n+2}}} \|\varphi_2\|_{L^{\frac{4n}{n+2}}}\|\varphi_3\|_{L^{\frac{4n}{n+2}}}\|\varphi_4\|_{L^{\frac{4n}{n+2}}}.\notag
\end{eqnarray*}
Then since $n\le 5$, the Sobolev imbedding yields that
\begin{eqnarray*}
\|\varphi_j\|_{L^{\frac{4n}{n+2}}}\le C\|\varphi_j\|,\quad j=1,2,3,4.
\end{eqnarray*}
Then (\ref{e:hls}) follows. This completes the proof.
\end{proof}

\subsection{Invertibility}
Let $D^2f_{\varepsilon}(z_{\xi})$ be the Hessian of $f_{\varepsilon}$ at $z_{\xi}$. Let $T_{z_{\xi}}(Z_{\varepsilon})$ be the tangent space of $Z_{\varepsilon}$ at $z_{\xi}$.
Define $L_{\varepsilon,\xi}:(T_{z_{\xi}}(Z_{\varepsilon}))^{\perp}\to H^1(\mathbb R^n)$ by
\begin{equation*}
\langle L_{\varepsilon,\xi}(v),w\rangle=D^2f_{\varepsilon}(z_{\xi})(v,w),\quad v\in(T_{z_{\xi}}(Z_{\varepsilon}))^{\perp} ,w\in H^1(\mathbb R^n).
\end{equation*}
Let $P_{\varepsilon,\xi}:H^1(\mathbb R^n)\to (T_{z_{\xi}}(Z_{\varepsilon}))^{\perp}$ be the orthogonal projection. We define
\begin{equation*}
\mathcal L_{\varepsilon,\xi}=P_{\varepsilon,\xi}L_{\varepsilon,\xi}:(T_{z_{\xi}}(Z_{\varepsilon}))^{\perp}\to(T_{z_{\xi}}(Z_{\varepsilon}))^{\perp}.
\end{equation*}
We now verify that $\mathcal L_{\varepsilon,\xi}$ is invertible on $(T_{z_{\xi}}(Z_{\varepsilon}))^{\perp}$.
\begin{prop}\label{p:invert}
For any fixed $\bar\delta>0$, there exists $\bar\varepsilon$ such that for all $0<\varepsilon<\bar\varepsilon$ and $|\xi|\le \bar\delta$
\begin{equation*}
|\langle L_{\varepsilon,\xi}w,w\rangle|\ge C\|w\|^2,\quad\forall w\in (T_{z_{\xi}}(Z_{\varepsilon}))^{\perp},
\end{equation*}
consequently,
\begin{equation*}
|\langle \mathcal L_{\varepsilon,\xi}w,w\rangle|\ge C\|w\|^2,\quad\forall w\in (T_{z_{\xi}}(Z_{\varepsilon}))^{\perp}.
\end{equation*}
Here $C$ is a constant depending only on $\bar\delta$ and $\bar\varepsilon$.
\end{prop}
\begin{proof}
1. Note that $T_{z_{\xi}}Z_{\varepsilon}={\rm span}\{\partial_{\xi_1}z_{\xi},\partial_{\xi_2}z_{\xi},\cdots,\partial_{\xi_n}z_{\xi}\}$. Thanks to Lemma \ref{l:pxpx}, we know that $\partial_{\xi_i}z_{\xi}$ is close to $-\partial_{x_i}z_{\xi}$ in $H^1(\mathbb R^n)$ as $\varepsilon\to 0$ and $|\xi|\le\bar\delta$. Let
\begin{equation*}
K_{\varepsilon,\xi}={\rm span}\{z_{\xi},\partial_{x_1}z_{\xi},\partial_{x_2}z_{\xi},\cdots,\partial_{x_n}z_{\xi}\}.
\end{equation*}
Hence it is sufficient to verify that there exists $\bar\varepsilon>0$ such that for all $0<\varepsilon<\bar\varepsilon$,
\begin{equation}\label{e:lzxzx}
\langle L_{\varepsilon,\xi}z_{\xi},z_{\xi}\rangle\le -C_1<0,
\end{equation}
and
\begin{equation}\label{e:phip}
\langle L_{\varepsilon,\xi}\phi,\phi\rangle\ge C_2\|\phi\|^2,\quad\forall \phi\in K_{\varepsilon,\xi}^{\perp},
\end{equation}
where $C_1,C_2$ are two constants depending only on $\bar\delta$ and $\bar\varepsilon$.

2. Recalling that $z_{\xi}$ is a mountain pass critical point of $F_{\varepsilon,\xi}$, we obtain that for any fixed $\varepsilon_1>0$ small, there exists a constant $c_1>0$ such that for all $0<\varepsilon\le \varepsilon_1$ and $|\xi|\le \bar\delta$,
\begin{equation*}
D^2F_{\varepsilon,\xi}(z_{\xi})[z_{\xi},z_{\xi}]<-c_1<0.
\end{equation*}
Therefore, from (\ref{e:nite}) we get
\begin{eqnarray*}
\langle L_{\varepsilon,\xi}z_{\xi},z_{\xi}\rangle&=&D^2F_{\varepsilon,\xi}(z_{\xi})[z_{\xi},z_{\xi}]+\int_{\mathbb R^n}(V(\varepsilon x)-V(\varepsilon\xi))z_{\xi}^2dx\notag\\
&\le&-c_1+c_2(\varepsilon |\nabla V(\varepsilon\xi)|+\varepsilon^2).
\end{eqnarray*}
Then (\ref{e:lzxzx}) follows.

3. We now prove (\ref{e:phip}). First, we have
\begin{equation*}
\langle L_{\varepsilon,\xi}\phi,\phi\rangle=D^2F_{\varepsilon,\xi}(z_{\xi})[\phi,\phi]+\int_{\mathbb R^n}(V(\varepsilon x)-V(\varepsilon\xi))\phi^2dx,\quad\forall \phi\in K_{\varepsilon,\xi}^{\perp}.
\end{equation*}
Since $z_{\xi}$ is a mountain pass critical point, it holds that
\begin{equation}\label{e:kperp}
D^2F_{\varepsilon,\xi}(z_{\xi})[\phi,\phi]\ge C_3\|\phi\|^2,\quad\forall \phi\in K_{\varepsilon,\xi}^{\perp}.
\end{equation}
Next we will borrow a cut-off function technique from \cite{AmMalSe2001}. For any $R>0$ sufficiently large, we choose a radial smooth cut-off function $\eta_1:\mathbb R^n\to\mathbb R$ such that
\begin{equation*}
\eta_1(x)=1\mbox{ for }|x|\le R,\,\eta_1(x)=0\mbox{ for } |x|\ge 2R, \mbox{ and } |\nabla \eta_1(x)|\le \frac{2}{R}\mbox{ for }R\le |x|\le 2R.
\end{equation*}
Let $\eta_2=1-\eta_1$. Let $\phi_i=\eta_i\phi$ ($i=1,2$). Then
\begin{eqnarray}\label{e:p123}
\|\phi\|^2&=&\|\phi_1\|^2+\|\phi_2\|^2+2\int_{\mathbb R^n}(\phi_1\phi_2+\nabla \phi_1\cdot\nabla\phi_2)dx\notag\\
&=&\|\phi_1\|^2+\|\phi_2\|^2+2\int_{\mathbb R^n}\eta_1\eta_2(\phi^2+|\nabla \phi|^2)dx+o_R(1)\|\phi\|^2.
\end{eqnarray}
By the definition of $\phi_i$,
\begin{equation}\label{e:st123}
\langle L_{\varepsilon,\xi}\phi,\phi\rangle
=\langle L_{\varepsilon,\xi}\phi_1,\phi_1\rangle+\langle L_{\varepsilon,\xi}\phi_2,\phi_2\rangle+2\langle L_{\varepsilon,\xi}\phi_1,\phi_2\rangle
:=T_1+T_2+T_3.
\end{equation}

4. First, we estimate $T_1$. Note that
\begin{equation*}
T_1=\langle L_{\varepsilon,\xi}\phi_1,\phi_1\rangle=D^2F_{\varepsilon,\xi}(z_{\xi})[\phi_1,\phi_1]+\int_{\mathbb R^n}(V(\varepsilon x)-V(\varepsilon\xi))\phi_1^2dx.
\end{equation*}
We split $\phi_1=\bar\phi_1+\psi$ where $\bar\phi_1\in K_{\varepsilon,\xi}^{\perp}$ and $\psi\in K_{\varepsilon,\xi}$.
Hence
\begin{equation*}
\psi=\langle \phi_1,z_{\xi}\rangle \|z_{\xi}\|^{-2}z_{\xi}+\sum_{i=1}^n\langle \phi_1,\partial_{x_i}z_{\xi}\rangle \|\partial_{x_i}z_{\xi}\|^{-2}\partial_{x_i}z_{\xi}.
\end{equation*}
It follows that
\begin{equation*}
D^2F_{\varepsilon,\xi}(z_{\xi})[\phi_1,\phi_1]
=D^2F_{\varepsilon,\xi}(z_{\xi})[\bar\phi_1,\bar\phi_1]
+D^2F_{\varepsilon,\xi}(z_{\xi})[\psi,\psi]+2D^2F_{\varepsilon,\xi}(z_{\xi})[\bar\phi_1,\psi].
\end{equation*}
(\ref{e:kperp}) implies that
\begin{equation}\label{e:kperp1}
D^2F_{\varepsilon,\xi}(z_{\xi})[\bar\phi_1,\bar\phi_1]\ge C_4\|\bar\phi_1\|^2.
\end{equation}
By $\phi\in K_{\varepsilon,\xi}^{\perp}$, it holds that
\begin{eqnarray*}
\langle \phi_1,z_{\xi}\rangle=-\langle\phi_2,z_{\xi}\rangle
=-\int_{\mathbb R^n}\eta_2(x)\phi(x) z_{\xi}(x)dx-\int_{\mathbb R^n}\nabla(\eta_2\phi)\cdot\nabla z_{\xi}dx.
\end{eqnarray*}
Since $\eta_2(x)=0$ for $|x|<R$ and $z_{\xi}(x)\to 0$ exponentially as $|x|\ge R\to\infty$ for $|\xi|\le \bar\delta$, we have that $\langle \phi_1,z_{\xi}\rangle=o_R(1)\|\phi\|$. A same estimate shows that $\langle \phi_1,\partial_{x_i}z_{\xi}\rangle=o_R(1)\|\phi\|$. Then
\begin{equation}\label{e:psiorphi}
\|\psi\|=o_R(1)\|\phi\|.
\end{equation}
Direct calculations yield that
\begin{eqnarray}\label{e:dfpsi}
D^2F_{\varepsilon,\xi}(z_{\xi})[\psi,\psi]&=&\|\psi\|^2+V(\varepsilon\xi)\int_{\mathbb R^n}\psi^2
-2\int_{\mathbb R^n}\int_{\mathbb R^n}\frac{z_{\xi}(x)\psi(x)z_{\xi}(y)\psi(y)}{|x-y|^{n-2}}dxdy\notag\\
&&\quad\quad-\int_{\mathbb R^n}\int_{\mathbb R^n}\frac{z_{\xi}^2(x)\psi^2(y)}{|x-y|^{n-2}}dxdy
\end{eqnarray}
Using Lemma \ref{l:hls}, we have
\begin{eqnarray}\label{e:dfpsi2}
\left|\int_{\mathbb R^n}\int_{\mathbb R^n}\frac{z_{\xi}(x)\psi(x)z_{\xi}(y)\psi(y)}{|x-y|^{n-2}}dxdy\right|
\le C\|\psi\|^2,
\end{eqnarray}
and
\begin{equation}\label{e:dfpsi3}
\left|\int_{\mathbb R^n}\int_{\mathbb R^n}\frac{z_{\xi}^2(x)\psi^2(y)}{|x-y|^{n-2}}dxdy\right|\le C\|\psi\|^2.
\end{equation}
Therefore, from (\ref{e:psiorphi}-\ref{e:dfpsi3}), we obtain
\begin{equation}\label{e:dfpsi4}
D^2F_{\varepsilon,\xi}(z_{\xi})[\psi,\psi]=o_R(1)\|\phi\|^2.
\end{equation}
Similarly, estimate
\begin{eqnarray}\label{e:dfpsiphi}
D^2F_{\varepsilon,\xi}(z_{\xi})[\bar\phi_1,\psi]&=&\langle\bar\phi_1,\psi\rangle+V(\varepsilon\xi)\int_{\mathbb R^n}\bar\phi_1\psi
-2\int_{\mathbb R^n}\int_{\mathbb R^n}\frac{z_{\xi}(x)\bar\phi_1(x)z_{\xi}(y)\psi(y)}{|x-y|^{n-2}}dxdy\notag\\
&&\quad\quad-\int_{\mathbb R^n}\int_{\mathbb R^n}\frac{z_{\xi}^2(x)\bar\phi_1(y)\psi(y)}{|x-y|^{n-2}}dxdy\notag\\
&=&o_{R}(1)\|\bar\phi_1\|\|\phi\|=o_R(1)\|\phi\|^2.
\end{eqnarray}
Combining (\ref{e:kperp1}), (\ref{e:dfpsi4}) and (\ref{e:dfpsiphi}), we get
\begin{equation}\label{e:dfpp}
D^2F_{\varepsilon,\xi}(z_{\xi})[\phi_1,\phi_1]\ge C\|\phi_1\|^2+o_R(1)\|\phi\|^2.
\end{equation}
On the other hand, for $|\xi|\le \bar\delta$,
\begin{eqnarray}\label{e:vvphi}
&&\int_{\mathbb R^n}|V(\varepsilon x)-V(\varepsilon\xi)|\phi_1^2dx\\
&\le&C_8\varepsilon\int_{\mathbb R^n}|x-\xi|\eta_1^2(x)\phi^2(x)dx\le C_9\varepsilon R\int_{\mathbb R^n}\eta_1^2(x)\phi^2(x)dx\le  C_9\varepsilon R\|\phi\|^2.\notag
\end{eqnarray}
Hence choosing $R=\varepsilon^{-\frac{1}{2}}$, (\ref{e:psiorphi}) and (\ref{e:dfpp}-\ref{e:vvphi}) yield that there exists $\bar\varepsilon$ small enough such that for all $\varepsilon\le \bar\varepsilon$,
\begin{equation}\label{e:t1}
T_1\ge C\|\phi_1\|^2+o_R(1)\|\phi\|^2.
\end{equation}

5. We now estimate $T_2$. A direct computation yields
\begin{eqnarray}\label{e:dfphi_2}
T_2=\langle L_{\varepsilon,\xi}\phi_2,\phi_2\rangle&=&\int_{\mathbb R^n}\left(|\nabla \phi_2|^2+(1+V(\varepsilon\xi))\phi_2^2\right)-\int_{\mathbb R^n}\int_{\mathbb R^n}\frac{z_{\xi}^2(x)\phi_2^2(y)}{|x-y|^{n-2}}dxdy.\notag\\
&&\quad-2\int_{\mathbb R^n}\int_{\mathbb R^n}\frac{z_{\xi}(x)\phi_2(x)z_{\xi}(y)\phi_2(y)}{|x-y|^{n-2}}dxdy.
\end{eqnarray}
Since $\inf_{x\in \mathbb R^n}(1+V(x))>0$, it follows that for $|\xi|\le \bar\delta$ and $\varepsilon$ small enough,
\begin{equation*}
\int_{\mathbb R^n}\left(|\nabla \phi_2|^2+(1+V(\varepsilon\xi))\phi_2^2\right)\ge C_9\|\phi_2\|^2.
\end{equation*}
Using Lemma \ref{l:hls} and $\eta_2=0$ in $B_R(0)$, we have
\begin{eqnarray*}
&&\left|\int_{\mathbb R^n}\int_{\mathbb R^n}\frac{z_{\xi}(x)\phi_2(x)z_{\xi}(y)\phi_2(y)}{|x-y|^{n-2}}dxdy\right|\\
&&\le C_5\|z_{\xi}\phi_2\|_{L^\frac{2n}{n+2}}^2\le C_5\|z_{\xi}\eta_2\|^{2}_{L^{n}}\|\phi\|^2\le o_R(1)\|\phi\|^2.\notag
\end{eqnarray*}
Furthermore, we estimate
\begin{eqnarray*}
&&\left|\int_{\mathbb R^n}\int_{\mathbb R^n}\frac{z_{\xi}^2(x)\phi_2^2(y)}{|x-y|^{n-2}}dxdy\right|
\le\int_{B_R^c}\int_{\mathbb R^n}\frac{z_{\xi}^2(x)}{|x-y|^{n-2}}\phi^2(y) dx dy\notag\\
&=&\int_{B_R^c}\int_{B_{R/2}}\frac{z_{\xi}^2(x)}{|x-y|^{n-2}}\phi^2(y) dx dy+\int_{B_R^c}\int_{B_{R/2}^c}\frac{z_{\xi}^2(x)}{|x-y|^{n-2}}\phi^2(y) dx dy\notag\\
&\le&C_{10}R^{-(n-2)}\|\phi\|^2+\int_{B_R^c}\phi^2(y)\left(\int_{B_{R/2}^c\cap B_1(y)}\frac{z_{\xi}^2(x)}{|x-y|^{n-2}}dx\right)dy\notag\\
&&\quad\quad+\int_{B_R^c}\phi^2(y)\left(\int_{B_{R/2}^c\cap B_1^c(y)}\frac{z_{\xi}^2(x)}{|x-y|^{n-2}}dx\right)dy\notag\\
&\le&C_{10}R^{-(n-2)}\|\phi\|^2+C_{11}e^{-R}\|\phi\|^2=o_R(1)\|\phi\|^2.
\end{eqnarray*}
Therefore, choosing $R$ large enough, we obtain that
\begin{equation}\label{e:t2}
T_2\ge C_{12}\|\phi_2\|^2+o_R(1)\|\phi\|^2.
\end{equation}

6. By a similar argument as in step 5, we get
\begin{eqnarray}\label{e:t3}
T_3\ge C_{14}\int_{\mathbb R^n}\eta_1\eta_2(|\nabla \phi|^2+\phi^2)dx+o_R(1)\|\phi\|^2.
\end{eqnarray}

7. Combining (\ref{e:t1}), (\ref{e:t2}), (\ref{e:t3}), (\ref{e:st123}) and (\ref{e:p123}), we finally prove (\ref{e:phip}). This completes the proof.
\end{proof}

\subsection{Lyapunov-Schmidt reduction}
In this subsection, we will show that there exists $w\in (T_{z_{\xi}}Z_{\varepsilon})^{\perp}$ such that
\begin{equation}\label{e:dfzw}
P_{\varepsilon,\xi}Df_{\varepsilon}(z_{\xi}+w)=0.
\end{equation}
Using expansion, we get
\begin{equation*}
Df_{\varepsilon}(z_{\xi}+w)=Df_{\varepsilon}(z_{\xi})+D^2f_{\varepsilon}(z_{\xi})[w]+\mathcal R(z_{\xi},w),
\end{equation*}
where
\begin{equation*}
\begin{array}{llcll}
                        \mathcal R(z_{\xi},w) & : & H^1(\mathbb R^n) & \to & \mathbb R \\
                          &   & \varphi & \to & \int_{\mathbb R^n}R(z_{\xi},w)\varphi dx.
                      \end{array}
\end{equation*}
Here $R(z_{\xi},w)$ is a high order nonlocal term given by
\begin{equation*}
R(z_{\xi},w)=-(I_2\ast w^2)z_{\xi}-2(I_2\ast(z_{\xi}w))w-(I_2\ast w^2)w.
\end{equation*}
Thus (\ref{e:dfzw}) becomes
\begin{equation}\label{e:dfexp}
\mathcal L_{\varepsilon,\xi} w+P_{\varepsilon,\xi}Df_{\varepsilon}(z_{\xi})+P_{\varepsilon}\mathcal R(z_{\xi},w)=0,\quad w\in (T_{z_{\xi}}(Z_{\varepsilon}))^{\perp}.
\end{equation}
Thanks to Proposition \ref{p:invert}, Equation (\ref{e:dfexp}) is equivalent to
\begin{equation}\label{e:defn}
w=-\mathcal L_{\varepsilon,\xi}^{-1}(P_{\varepsilon,\xi}Df_{\varepsilon}(z_{\xi})+P_{\varepsilon}\mathcal R(z_{\xi},w)):=N_{\varepsilon,\xi}(w).
\end{equation}
Next we prove that for proper $\varepsilon,\xi$, the map $N_{\varepsilon,\xi}$ is a contraction.
\begin{lemma}\label{l:rcontraction}
For all $w_1,w_2\in B_1\subset H^1(\mathbb R^n)$,
\begin{equation}\label{e:rw1w2}
\|\mathcal R(z_{\xi},w_2)-\mathcal R(z_{\xi},w_1)\|\le C(\|w_1\|+\|w_2\|)\|w_2-w_1\|,
\end{equation}
where $C$ is a constant independent on $w_1$ and $w_2$, $B_1$ is the unit ball in $H^1(\mathbb R^n)$.
\end{lemma}
\begin{proof}
By Lemma \ref{l:hls}, we have that for all $\varphi\in H^1(\mathbb R^n)$,
\begin{eqnarray*}
&&|\mathcal R(z_{\xi},w_2)(\varphi)-\mathcal R(z_{\xi},w_1)(\varphi)|\\
&\le&\int_{\mathbb R^n}[I_2\ast|w_2^2-w_1^2|]|z_{\xi}\varphi| dx+2\int_{\mathbb R^n}|(I_2\ast(z_{\xi}w_2))w_2-(I_2\ast(z_{\xi}w_1))w_1||\varphi|dx\notag\\
&&\quad\quad\quad+\int_{\mathbb R^n}|(I_2\ast w_2^2)w_2-(I_2\ast w_1^2)w_1||\varphi|dx\notag\\
&:=&C(\|w_1\|+\|w_2\|)\|w_2-w_1\|\|\varphi\|.\notag
\end{eqnarray*}
This yields (\ref{e:rw1w2}).
\end{proof}
\begin{remark}\label{r:rsmall}
As a consequence, it holds that $\|\mathcal R(z_{\xi},w)\|=O(\|w\|^2)$.
\end{remark}

\begin{lemma}\label{l:ncontraction}
There exists a small ball $B_{\delta}\subset (T_{z_{\xi}}(Z_{\varepsilon}))^{\perp}$ such that $N_{\varepsilon,\xi}$ maps $B_{\delta}$ into itself for $0<\varepsilon\le \bar\varepsilon$ and $|\xi|\le \bar\delta$. Moreover, for all $w_1,w_2\in B_{\delta}$,
\begin{equation*}
\|N_{\varepsilon,\xi}(w_2)-N_{\varepsilon,\xi}(w_1)\|\le C(\|w_1\|+\|w_2\|)\|w_2-w_1\|,
\end{equation*}
where $C$ is a constant independent on $w_1$ and $w_2$. In particular, $N_{\varepsilon,\xi}$ is a contraction map on $B_{\delta}$.
\end{lemma}
\begin{proof}
From Lemma \ref{l:dfvz}, Remark \ref{r:rsmall} and (\ref{e:defn}), we have that
\begin{eqnarray}\label{e:nwev}
\|N_{\varepsilon,\xi}(w)\|\le C\|Df_{\varepsilon}(z_{\xi})\|+O(\|w\|^2)\le C(\varepsilon \nabla V(\varepsilon\xi)+O(\varepsilon^2))+O(\|w\|^2).
\end{eqnarray}
Hence there exists a small $\delta>0$ (i.e. $\delta=\varepsilon^{1/3}$) such that $N_{\varepsilon,\xi}$ maps $B_{\delta}$ into itself $0<\varepsilon\le \bar\varepsilon$ and $|\xi|\le \bar\delta$.

Furthermore, by Lemma \ref{l:rcontraction}, it holds that for $w_1,w_2\in B_{\delta}$,
\begin{eqnarray*}
\|N_{\varepsilon,\xi}(w_2)-N_{\varepsilon,\xi}(w_1)\|&\le& \|-\mathcal L_{\varepsilon,\xi}^{-1}(\mathcal R(z_{\xi},w_2)-\mathcal R(z_{\xi},w_1))\|\\
&\le &C\|\mathcal R(z_{\xi},w_2)-\mathcal R(z_{\xi},w_1)\|\le C(\|w_1\|+\|w_2\|)\|w_2-w_1\|.\notag
\end{eqnarray*}
This completes the proof.
\end{proof}

With those preparations at hand, we have the following result.
\begin{prop}\label{p:dfw0}
For $0<\varepsilon\le \bar\varepsilon$ and $|\xi|\le \bar\delta$, there exists a unique $w=w(\varepsilon,\xi)\in (T_{z_{\xi}}(Z_{\varepsilon}))^{\perp}$ of class $C^1$ with respect to $\xi$ such that $Df_{\varepsilon}(z_{\xi}+w)\in T_{z_{\xi}}(Z_{\varepsilon})$. Moreover, the functional $\Phi_{\varepsilon}(\xi):=f_{\varepsilon}(z_{\xi}+w(\varepsilon,\xi))$ has the same regularity as $w$ and satisfies:
\begin{equation*}
\nabla \Phi(\xi_0)=0\Rightarrow Df_{\varepsilon}(z_{\xi_0}+w(\varepsilon,\xi_0))=0.
\end{equation*}
\end{prop}
\begin{proof}
Since $N_{\varepsilon,\xi}$ is a contraction map on $B_{\delta}$ ($\delta$ is chosen as in Lemma \ref{l:ncontraction}) for $0<\varepsilon\le \bar\varepsilon$ and $|\xi|\le \bar\delta$, there exists a unique $w$ such that $w=N_{\varepsilon,\xi}(w)$. Furthermore, for fixed $\varepsilon>0$, applying the implicit function theorem to the map $(\xi,w)\to P_{\varepsilon,\xi}Df_{\varepsilon}(z_{\xi}+w)$, it holds that $w(\varepsilon,\xi)$ is of class $C^1$ with respect to $\xi$. Then standard argument as in \cite{AmbBad98,AmbBadCin1997} tells us that the critical points of $\Phi_{\varepsilon}=f_{\varepsilon}(z_{\xi}+w(\varepsilon,\xi))$ give rise to critical points of $f_{\varepsilon}$.
\end{proof}

\begin{remark}\label{r:wve}
From (\ref{e:nwev}), we see that
\begin{equation}\label{e:wve}
\|w\|\le C(\varepsilon V(\varepsilon\xi)+O(\varepsilon^2)).
\end{equation}
\end{remark}

We now verify that $\Phi_{\varepsilon}$ in fact is a perturbation of some function of $V$.

By the definition, it holds that
\begin{eqnarray*}
&&\Phi_{\varepsilon}(\xi)=\frac{1}{2}\|z_{\xi}+w(\varepsilon,\xi)\|^2+\frac{1}{2}\int_{\mathbb R^n}V(\varepsilon x)(z_{\xi}+w(\varepsilon,\xi))^2\\
&&\quad\quad\quad\quad\quad\quad\quad\quad\quad-\frac{1}{4}\int_{\mathbb R^n}(I_2\ast (z_{\xi}+w(\varepsilon,\xi))^2)(z_{\xi}+w(\varepsilon,\xi))^2.\notag
\end{eqnarray*}
Recalling that $z_{\xi}$ solves (\ref{e:namain2}), we have that
\begin{eqnarray*}
\|z_{\xi}\|^2=-V(\varepsilon\xi)\int_{\mathbb R^n}z_{\xi}^2+\int_{\mathbb R^n}(I_2\ast z_{\xi}^2)z_{\xi}^2,
\end{eqnarray*}
and
\begin{eqnarray*}
\langle z_{\xi},w\rangle=-V(\varepsilon\xi)\int_{\mathbb R^n}z_{\varepsilon}w+\int_{\mathbb R^n}(I_2\ast z_{\xi}^2)z_{\xi}w.
\end{eqnarray*}
Then
\begin{eqnarray*}
\Phi_{\varepsilon}(\xi)&=&\frac{1}{4}\int_{\mathbb R^n}[I_2\ast z_{\xi}^2]z_{\xi}^2+\frac{1}{2}\int_{\mathbb R^n}[V(\varepsilon x)-V(\varepsilon\xi)]z_{\xi}^2+\int_{\mathbb R^n}[V(\varepsilon x)-V(\varepsilon\xi)]z_{\xi}w\notag\\
&&+\frac{1}{2}\int_{\mathbb R^n}V(\varepsilon x)w^2+\frac{1}{2}\|w\|^2-\int_{\mathbb R^n}[I_2\ast (z_{\xi}w)]z_{\xi}w\notag\\
&&-\frac{1}{2}\int_{\mathbb R^n}[I_2\ast z_{\xi}^2]w^2-\int_{\mathbb R^n}[I_2\ast (z_{\xi}w)]w^2-\frac{1}{4}\int_{\mathbb R^n}[I_2\ast w^2]w^2.
\end{eqnarray*}
By the definition of $z_{\xi}$ (see (\ref{e:dfzx})), we obtain that
\begin{eqnarray*}
\int_{\mathbb R^n}[I_2\ast z_{\xi}^2]z_{\xi}^2&=&\int_{\mathbb R^n}\int_{\mathbb R^n}\frac{(\alpha(\varepsilon\xi))^2U^2(\beta(\varepsilon\xi)(x-\xi))(\alpha(\varepsilon\xi))^2U^2(\beta(\varepsilon\xi)(y-\xi))}{|x-y|^{n-2}}dxdy\\
&=&(\alpha(\varepsilon\xi))^4(\beta(\varepsilon\xi))^{-(n+2)}\int_{\mathbb R^n}\int_{\mathbb R^n}\frac{U^2(\bar x)U^2(\bar y)}{|\bar x-\bar y|^{n-2}}d\bar x d\bar y
=C_0(1+V(\varepsilon\xi))^{3-\frac{n}{2}},\notag
\end{eqnarray*}
where
\begin{equation*}
C_0=\int_{\mathbb R^n}\int_{\mathbb R^n}\frac{U^2(\bar x)U^2(\bar y)}{|\bar x-\bar y|^{n-2}}d\bar x d\bar y.
\end{equation*}
\begin{remark}
By assumption $n=3,4,5$, we have that $3-\frac{n}{2}>0$.
\end{remark}

For simplicity, set
\begin{eqnarray*}
\Gamma_{\varepsilon}(\xi)=\frac{1}{2}\int_{\mathbb R^n}[V(\varepsilon x)-V(\varepsilon\xi)]z_{\xi}^2+\int_{\mathbb R^n}[V(\varepsilon x)-V(\varepsilon\xi)]z_{\xi}w,
\end{eqnarray*}
and
\begin{eqnarray*}
\Psi_{\varepsilon}(\xi)&=&\frac{1}{2}\int_{\mathbb R^n}V(\varepsilon x)w^2+\frac{1}{2}\|w\|^2-\int_{\mathbb R^n}[I_2\ast (z_{\xi}w)]z_{\xi}w\\
&&-\frac{1}{2}\int_{\mathbb R^n}[I_2\ast z_{\xi}^2]w^2-\int_{\mathbb R^n}[I_2\ast (z_{\xi}w)]w^2-\frac{1}{4}\int_{\mathbb R^n}[I_2\ast w^2]w^2.\notag
\end{eqnarray*}
Therefore,
\begin{eqnarray}\label{e:dfofPhi}
\Phi_{\varepsilon}(\xi)=C_1(1+V(\varepsilon\xi))^{3-\frac{n}{2}}+\Gamma_{\varepsilon}(\xi)+\Psi_{\varepsilon}(\xi),
\end{eqnarray}
where $C_1=\frac{1}{4}C_0$.

We need to estimate $\Gamma_{\varepsilon}$ and $\Psi_{\varepsilon}$. Before that, we need to compute $\nabla_{\xi} w$.
\begin{lemma}\label{l:nwx}
For $0<\varepsilon\le \bar\varepsilon$ and $|\xi|\le \bar\delta$, it holds that
\begin{equation*}
|\nabla_{\xi}w|\le C\left(\varepsilon|\nabla V(\varepsilon\xi)|+O(\varepsilon^2)\right),
\end{equation*}
where $C$ is a constant depending on $\bar\delta$ and $\bar\varepsilon$.
\end{lemma}
\begin{proof}
By (\ref{e:dfexp}) and Proposition \ref{p:dfw0}, it holds that
\begin{equation*}
\langle\mathcal L_{\varepsilon,\xi} w,\varphi\rangle+\langle Df_{\varepsilon}(z_{\xi}),\varphi\rangle+\langle\mathcal R(z_{\xi},w),\varphi\rangle=0,\quad\forall\varphi\in (T_{z_{\xi}}(Z_{\varepsilon}))^{\perp}.
\end{equation*}
Then using the definition of $\mathcal L_{\varepsilon,\xi}$ and $DF_{\varepsilon,\xi}(z_{\xi})=0$, we have
\begin{eqnarray*}
0&=&\langle w,\varphi\rangle+\int_{\mathbb R^n}V(\varepsilon x)w\varphi dx-\int_{\mathbb R^n}\int_{\mathbb R^n}\frac{z_{\xi}^2(x)w(y)\varphi(y)}{|x-y|^{n-2}}dxdy\notag\\
&&-2\int_{\mathbb R^n}\int_{\mathbb R^n}\frac{z_{\xi}(x)w(x)z_{\xi}(y)\varphi(y)}{|x-y|^{n-2}}dxdy\notag\\
&&+\int_{\mathbb R^n}[V(\varepsilon x)-V(\varepsilon\xi)]z_{\xi}\varphi dx+\int_{\mathbb R^n}R(z_{\xi},w)\varphi dx.
\end{eqnarray*}
It follows that
\begin{eqnarray*}
0&=&\langle \partial_{\xi_i}w,\varphi\rangle+\int_{\mathbb R^n}V(\varepsilon x)\partial_{\xi_i}w\varphi dx-\int_{\mathbb R^n}\int_{\mathbb R^n}\frac{z_{\xi}^2(x)\partial_{\xi_i}w(y)\varphi(y)}{|x-y|^{n-2}}dxdy\notag\\
&&-2\int_{\mathbb R^n}\int_{\mathbb R^n}\frac{z_{\xi}(x)\partial_{\xi_i}w(x)z_{\xi}(y)\varphi(y)}{|x-y|^{n-2}}dxdy\notag\\
&&-2\int_{\mathbb R^n}\int_{\mathbb R^n}\frac{z_{\xi}\partial_{\xi_i}z_{\xi}(x)w(y)\varphi(y)}{|x-y|^{n-2}}dxdy
-2\int_{\mathbb R^n}\int_{\mathbb R^n}\frac{\partial_{\xi_i}z_{\xi}(x)w(x)z_{\xi}(y)\varphi(y)}{|x-y|^{n-2}}dxdy\notag\\
&&-2\int_{\mathbb R^n}\int_{\mathbb R^n}\frac{z_{\xi}(x)w(x)\partial_{\xi_i}z_{\xi}(y)\varphi(y)}{|x-y|^{n-2}}dxdy
-\varepsilon (\partial_{x_i}V)(\varepsilon\xi)\int_{\mathbb R^n}z_{\xi}\varphi dx\notag\\
&&+\int_{\mathbb R^n}[V(\varepsilon x)-V(\varepsilon\xi)]\partial_{\xi_i}z_{\xi}\varphi dx+\partial_{\xi_i}\left(\int_{\mathbb R^n}R(z_{\xi},w)\varphi dx\right).
\end{eqnarray*}
Again by the definition of $\mathcal L_{\varepsilon,\xi}$, we get that
\begin{eqnarray}\label{e:pxwlv}
&&\langle\mathcal L_{\varepsilon,\xi}(\partial_{\xi_i}w),\varphi\rangle\notag\\
&&=2\int_{\mathbb R^n}\int_{\mathbb R^n}\frac{z_{\xi}\partial_{\xi_i}z_{\xi}(x)w(y)\varphi(y)}{|x-y|^{n-2}}dxdy
+2\int_{\mathbb R^n}\int_{\mathbb R^n}\frac{\partial_{\xi_i}z_{\xi}(x)w(x)z_{\xi}(y)\varphi(y)}{|x-y|^{n-2}}dxdy\notag\\
&&\quad+2\int_{\mathbb R^n}\int_{\mathbb R^n}\frac{z_{\xi}(x)w(x)\partial_{\xi_i}z_{\xi}(y)\varphi(y)}{|x-y|^{n-2}}dxdy
+\varepsilon (\partial_{x_i}V)(\varepsilon\xi)\int_{\mathbb R^n}z_{\xi}\varphi dx\notag\\
&&\quad-\int_{\mathbb R^n}[V(\varepsilon x)-V(\varepsilon\xi)]\partial_{\xi_i}z_{\xi}\varphi dx-\partial_{\xi_i}\left(\int_{\mathbb R^n}R(z_{\xi},w)\varphi dx\right).
\end{eqnarray}
Next we estimate the right side of the last equation term by term.

Using Lemma \ref{l:hls}, we find that
\begin{eqnarray}\label{e:hlsem}
\left|\int_{\mathbb R^n}\int_{\mathbb R^n}\frac{z_{\xi}\partial_{\xi_i}z_{\xi}(x)w(y)\varphi(y)}{|x-y|^{n-2}}dxdy\right|
\le C\|w\| \|\varphi\|,
\end{eqnarray}
and
\begin{multline}\label{e:hlsem2}
\left|\int_{\mathbb R^n}\int_{\mathbb R^n}\frac{\partial_{\xi_i}z_{\xi}(x)w(x)z_{\xi}(y)\varphi(y)}{|x-y|^{n-2}}dxdy\right|+\left|\int_{\mathbb R^n}\int_{\mathbb R^n}\frac{z_{\xi}(x)w(x)\partial_{\xi_i}z_{\xi}(y)\varphi(y)}{|x-y|^{n-2}}dxdy\right|\\
\le C\|w\| \|\varphi\|.
\end{multline}
Similarly, by Lemma \ref{l:hls}, it holds that
\begin{eqnarray}\label{e:hlsem3}
&&\left|\partial_{\xi_i}\left(\int_{\mathbb R^n}R(z_{\xi},w)\varphi dx\right)\right|\\
&\le&\left|\int_{\mathbb R^n}\left(2[I_2\ast (w\partial_{\xi_i}w)]z_{\xi}+[I_2\ast w^2]\partial_{\xi_i}z_{\xi}\right)\varphi\right|\notag\\
&&+\left|\int_{\mathbb R^n}\left(2[I_2\ast (\partial_{\xi_i}z_{\xi}w)]w+2[I_2\ast (z_{\xi}\partial_{\xi_i}w)]w+[I_2\ast (z_{\xi}w)]\partial_{\xi_i}w\right)\varphi\right|\notag\\
&&+\left|\int_{\mathbb R^n}\left(2[I_2\ast (w\partial_{\xi_i}w)]w+[I_2\ast w^2]\partial_{\xi_i}w\right)\varphi\right|\le C\|w\|\|\varphi\|\|\partial_{\xi_i}w\|.
\end{eqnarray}
The H\"older inequality yields that
\begin{eqnarray}\label{e:hlsem4}
\left|\varepsilon (\partial_{x_i}V)(\varepsilon\xi)\int_{\mathbb R^n}z_{\xi}\varphi dx\right|\le C\varepsilon|\nabla V(\varepsilon\xi)|\|\varphi\|.
\end{eqnarray}
Combining the H\"older inequality, Lemma \ref{l:decayu'} and (\ref{e:nite}), we have
\begin{eqnarray}\label{e:hlsem5}
\left|\int_{\mathbb R^n}[V(\varepsilon x)-V(\varepsilon\xi)]\partial_{\xi_i}z_{\xi}\varphi dx\right|\le C(\varepsilon|\nabla V(\varepsilon\xi)|+\varepsilon^2)\|\varphi\|.
\end{eqnarray}
Note that $\mathcal L_{\varepsilon,\xi}$ is invertible in $(T_{z_{\varepsilon}}(Z_{\varepsilon}))^{\perp}$. Then from (\ref{e:pxwlv}-\ref{e:hlsem5}) and (\ref{e:wve}), we obtain that
\begin{eqnarray*}
\|\partial_{\xi_i}w\|\le C(\varepsilon|\nabla V(\varepsilon\xi)|+\varepsilon^2),
\end{eqnarray*}
where $C$ is a constant depending on $\bar\delta$ and $\bar\varepsilon$. This completes the proof.
\end{proof}
\begin{lemma}\label{l:psigammasmall}
For $0<\varepsilon\le\bar\varepsilon$ and $|\xi|\le \bar \delta$, it holds that
\begin{eqnarray}\label{e:gammasmall}
|\Gamma_{\varepsilon}(\xi)|+|\Psi_{\varepsilon}(\xi)|\le C(\varepsilon|\nabla V(\varepsilon\xi)|+\varepsilon^2),
\end{eqnarray}
and
\begin{eqnarray}\label{e:nablagamma}
|\nabla \Gamma_{\varepsilon}(\xi)|+|\nabla \Psi_{\varepsilon}(\xi)|\le C\varepsilon^2.
\end{eqnarray}
\end{lemma}
\begin{proof}
Using Lemma \ref{l:hls} and a similar argument as in the proof of (\ref{e:nite}), we can obtain (\ref{e:gammasmall}). The details are omitted here. We now focus on the verification of (\ref{e:nablagamma}) which is more complicated.

1. By Taylor expansion of $V$, we get that
\begin{eqnarray*}
&&\int_{\mathbb R^n}[V(\varepsilon x)-V(\varepsilon\xi)]z_{\xi}^2\\
&=&\varepsilon\int_{\mathbb R^n}\nabla V(\varepsilon\xi)\cdot (x-\xi)z_{\xi}^2dx+\varepsilon^2\int_{\mathbb R^n}D^2V(\varepsilon\xi+\vartheta\varepsilon(x-\xi))[x-\xi,x-\xi]z_{\xi}^2dx\notag\\
&=&\varepsilon\int_{\mathbb R^n}\nabla V(\varepsilon\xi)\cdot y (\alpha(\varepsilon\xi)U(\beta(\varepsilon\xi)y)^2dy+\varepsilon^2\int_{\mathbb R^n}D^2V(\varepsilon\xi+\vartheta\varepsilon(x-\xi))[x-\xi,x-\xi]z_{\xi}^2dx\notag\\
&=&\varepsilon^2\int_{\mathbb R^n}D^2V(\varepsilon\xi+\vartheta\varepsilon(x-\xi))[x-\xi,x-\xi]z_{\xi}^2dx\notag,
\end{eqnarray*}
where $\vartheta\in (0,1)$. Here the third equality holds because $\nabla V(\varepsilon\xi)\cdot y (\alpha(\varepsilon\xi)U(\beta(\varepsilon\xi)y)^2$ is odd with respect to $y$ in $\mathbb R^n$. Since $V\in C_{b}^2(\mathbb R^n)$, we have that
\begin{eqnarray}\label{e:nablavz2}
&&\left|\partial_{\xi_i}\left(\int_{\mathbb R^n}[V(\varepsilon x)-V(\varepsilon\xi)]z_{\xi}^2dx\right)\right|\\
&\le& \varepsilon^2\left|\partial_{\xi_i}\left(\int_{\mathbb R^n}D^2V(\varepsilon\xi+\vartheta\varepsilon(x-\xi))[x-\xi,x-\xi]z_{\xi}^2dx\right)\right|\le C\varepsilon^2.\notag
\end{eqnarray}

On the other hand, by the H\"older inequality, we find
\begin{eqnarray*}
&&\left|\partial_{\xi_i}\left(\int_{\mathbb R^n}[V(\varepsilon x)-V(\varepsilon\xi)]z_{\xi}wdx\right)\right|\\
&\le&\varepsilon |\nabla V(\varepsilon\xi)|\int_{\mathbb R^n}|z_{\xi}w|dx+\int_{\mathbb R^n}|V(\varepsilon x)-V(\varepsilon\xi)||\partial_{\xi_i}z_{\xi}||w|dx\notag\\
&&+\int_{\mathbb R^n}|V(\varepsilon x)-V(\varepsilon\xi)||z_{\xi}||\partial_{\xi_i}w|dx\notag\\
&\le&C\varepsilon |\nabla V(\varepsilon\xi)|\|w\|
+\left(\int_{\mathbb R^n}|V(\varepsilon x)-V(\varepsilon\xi)|^2|\partial_{\xi_i}z_{\xi}|^2\right)^{\frac{1}{2}}\|w\|\notag\\
&&+\left(\int_{\mathbb R^n}|V(\varepsilon x)-V(\varepsilon\xi)|^2|z_{\xi}|^2\right)^{\frac{1}{2}}\|\partial_{\xi_i}w\|.\notag
\end{eqnarray*}
From (\ref{e:nite}), Lemma \ref{l:decayu'}, Remark \ref{r:wve} and Lemma \ref{l:nwx}, it holds that
\begin{eqnarray}\label{e:nablavzw}
\left|\nabla\left(\int_{\mathbb R^n}[V(\varepsilon x)-V(\varepsilon\xi)]z_{\xi}wdx\right) \right|\le C\varepsilon(\varepsilon+\|w\|+\|\nabla w\|)\le C\varepsilon^2.
\end{eqnarray}
Then from (\ref{e:nablavz2}) and (\ref{e:nablavzw}), it holds that
\begin{eqnarray}\label{e:nablagamma2}
|\nabla \Gamma_{\varepsilon}(\xi)|\le C\varepsilon^2.
\end{eqnarray}

2. We now estimate $|\nabla \Psi_{\varepsilon}|$. Compute
\begin{eqnarray*}
\left|\partial_{\xi_i}\left(\int_{\mathbb R^n}V(\varepsilon x)w^2dx+\|w\|^2\right)\right|
&\le& 2\int_{\mathbb R^n}V(\varepsilon x)|w||\partial_{\xi_i}w|dx+2|\langle w,\partial_{\xi_i}w\rangle|\notag\\
&\le& C\|w\|\|\partial_{\xi_i}w\|.
\end{eqnarray*}
Moreover, by Lemma \ref{l:hls}, we get that
\begin{eqnarray*}
&&\left|\partial_{\xi_i}\left(\int_{\mathbb R^n}[I_2\ast (z_{\xi}w)]z_{\xi}w\right)\right|
=2\int_{\mathbb R^n}\int_{\mathbb R^n}\frac{|[\partial_{\xi_i}z_{\xi}](x)w(x)||z_{\xi}(y)w(y)|}{|x-y|^{n-2}}dxdy\\
&&\quad\quad+2\int_{\mathbb R^n}\int_{\mathbb R^n}\frac{|z_{\xi}(x)[\partial_{\xi_i}w](x)||z_{\xi}(y)w(y)|}{|x-y|^{n-2}}dxdy
\le C\|w\|(\|w\|+\|\partial_{\xi_i}w\|),\notag
\end{eqnarray*}
and
\begin{eqnarray*}
&&\frac{1}{2}\left|\partial_{\xi_i}\int_{\mathbb R^n}[I_2\ast z_{\xi}^2](w(\varepsilon,\xi))^2\right|+\left|\partial_{\xi_i}\int_{\mathbb R^n}[I_2\ast (z_{\xi}w(\varepsilon,\xi))](w(\varepsilon,\xi))^2\right|\notag\\
&&+\left|\partial_{\xi_i}\frac{1}{4}\int_{\mathbb R^n}[I_2\ast (w(\varepsilon,\xi))^2](w(\varepsilon,\xi))^2\right|\le C\|w\|(\|w\|+\|\partial_{\xi_i}w\|).
\end{eqnarray*}
Hence by Lemma \ref{l:nwx} and Remark \ref{r:wve}, it holds that
\begin{eqnarray}\label{e:nablapsi}
|\nabla \Psi_{\varepsilon}(\xi)|\le C\|w\|(\|w\|+\|\partial_{\xi_i}w\|)\le C\varepsilon^2.
\end{eqnarray}
Then putting (\ref{e:nablagamma2}) and (\ref{e:nablapsi}) together, we obtain (\ref{e:nablagamma}). This completes the proof.
\end{proof}

\subsection{Proof of Theorem \ref{t:multiple}}

Let $M\subset \mathbb R^n$ be a non-empty set. We denote by $M_{\delta}$ its $\delta-$neighbourhood. The cup length $l(M)$ of $M$ is given by
\begin{eqnarray}\label{e:cupl}
l(M)=1+\sup\{k\in \mathbb N\,|\,\exists \Lambda_1,\Lambda_2,\cdots,\Lambda_k\in \breve{H}^*(M)\setminus 1,\,\Lambda_1\cup\Lambda_2\cup\cdots\cup\Lambda_k\ne 0\},
\end{eqnarray}
where $\breve{H}^*(M)$ is the Alexander cohomology of $M$ with real coefficients.

\begin{proof}[Proof of Theorem \ref{t:multiple}]
Choose $\bar R>0$ sufficiently large such that $M\subset B_{\bar R}$. $M$ is a non-degenerate critical manifold of $C_1(1+V)^{3-\frac{n}{2}}$ (recall that $M$ is a non-degenerate smooth critical manifold of $V$). Let
\begin{equation*}
h(\xi)=C_1(1+V(\xi))^{3-\frac{n}{2}}\quad\mbox{ and }\quad g(\xi)=\Phi_{\varepsilon}\left(\frac{\xi}{\varepsilon}\right).
\end{equation*}
Choose a $\delta$-neighbourhood $M_{\delta}$ of $M$ such that $M_{\delta}\subset B_{\bar R}$, so the set of critical points of $V$ in $M_{\delta}$ is $M$. From (\ref{e:dfofPhi}) and Lemma \ref{l:psigammasmall}, the function $\Phi_{\varepsilon}(\cdot/\varepsilon)$ is converges to $h(\cdot)$ in $C^1(\overline {M_{\delta}})$ as $\varepsilon\to 0$. Then \cite[Theorem 6.4 of Chapter II]{Chang93} yields that there exist at least $l(M)$ critical points of $g$ for $\varepsilon$ small enough.

Assume $\xi_k\in M_{\delta}$ such that $\xi_k/\varepsilon$ is a critical point of $\Phi_{\varepsilon}$. Then Proposition \ref{p:dfw0} implies that
\begin{equation*}
u_{\varepsilon,\xi_k}(x):=z_{\xi_k}\left(x-\frac{\xi_k}{\varepsilon}\right)+w(\varepsilon,\xi_k)
\end{equation*}
is a critical point of $f_{\varepsilon}$. Hence
\begin{equation*}
u_{\varepsilon,\xi_k}\left(\frac{x}{\varepsilon}\right)\simeq z_{\xi_k}\left(\frac{x-\xi_k}{\varepsilon}\right)
\end{equation*}
is a solution of Equation (\ref{e:namain}). When $\varepsilon\to 0$, $\xi_k$ converges to some point $\bar\xi_k\in M_{\delta}$. By Lemma \ref{l:psigammasmall} and (\ref{e:dfofPhi}), we conclude that $\bar\xi_k$ is a critical point of $V$. Since $\delta$ is arbitrary, we have that $\bar\xi_k\in M$. Therefore, $u_{\varepsilon,\xi_k}(x/\varepsilon)$ concentrates to a point of $M$. This completes the proof.
\end{proof}

\section*{Acknowledgement} The author is indebted to Minbo Yang for many useful discussions on Choquard equation. This work is supported by Zhejiang Provincial Science Foundation of China (No. LY18A010023), NSFC (No.11771386) and First Class Discipline of Zhejiang - A (Zhejiang University of Finance and Economics- Statistics).